\documentclass[11pt]{amsart}
\usepackage[margin=1in]{geometry}
\usepackage{bm}

\usepackage{mathpazo}
\usepackage{latexsym}
\usepackage{amssymb}
\usepackage{amsmath}
\usepackage{amsthm}
\usepackage{stmaryrd}
\usepackage{hyperref}
\usepackage[cmtip,all,matrix,arrow,tips,curve]{xy}
\usepackage[normalem]{ulem}

\newcommand{\doi}[1]{\textsc{doi}: \href{http://dx.doi.org/#1}{\nolinkurl{#1}}}
\numberwithin{equation}{subsection}
\newtheorem{theorem}{Theorem}[section]
\newtheorem{lemma}[theorem]{Lemma}
\newtheorem{proposition}[theorem]{Proposition}

\newtheorem*{theorem*}{Theorem}
\newtheorem*{prop*}{Proposition}

\theoremstyle{definition}
\newtheorem{remark}[theorem]{Remark}

\newtheorem{definition}[theorem]{Definition}

\DeclareMathOperator{\disc}{disc}
\DeclareMathOperator{\defo}{Def}
\DeclareMathOperator{\orth}{O}
\DeclareMathOperator{\sorth}{SO}
\DeclareMathOperator{\U}{U}

\DeclareMathOperator{\rank}{rank}

\def\shim{{\stack{Sh}}}
\def\triv{{\mathrm{triv}}}

\def\mk{{\stack R}}
\def\mcs{\stack{Cub}_2}
\def\mct{\stack{Cub}_3}
\def\lk{{L_{{\rm K3}}}}
\def\l2d{\ang{2d}}

\def\xx{{\mathbb X}}
\def\hh{{\mathbb H}}

\def\ball{{\mathbb B}}

\def\kk{{\mathbb K}}

\newcommand{\stack}[1]{{\sf #1}}
\newcommand{\tilstack}[1]{{\til{\stack #1}}}

\def\aff{{\mathbb A}}
\def\cx{{\mathbb C}}
\def\ff{{\mathbb F}}
\def\proj{{\mathbb P}}
\def\rat{{\mathbb Q}}
\def\real{{\mathbb R}}
\def\integ{{\mathbb Z}}
\def\gp{{\mathbb G}}

\def\mmu{\bm{\mu}}

\def\call{{\mathcal L}}
\def\calh{{\mathcal H}}
\def\calo{{\mathcal O}}
\def\calt{{\mathcal T}}

\DeclareMathOperator{\ts}{T}

\newcommand{\til}[1]{{\widetilde{#1}}}
\newcommand{\st}[1]{\{#1\}}
\newcommand{\ang}[1]{{{\langle #1 \rangle}}}
\newcommand{\rest}[1]{|_{#1}}
\newcommand{\invlim}[1]{\lim_{\stackrel{\leftarrow}{#1}}}
\newcommand{\floor}[1]{{\lfloor #1 \rfloor}}

\global\let\det\undefined
\DeclareMathOperator{\det}{det}
\global\let\hom\undefined
\DeclareMathOperator{\hom}{Hom}
\DeclareMathOperator{\id}{id}
\DeclareMathOperator{\aut}{Aut}

\DeclareMathOperator{\fil}{Fil}
\DeclareMathOperator{\gr}{Gr}
\DeclareMathOperator{\pic}{Pic}
\DeclareMathOperator{\ns}{NS}
\def\res{\bm{{\rm R}}}

\DeclareMathOperator{\SP}{Sp}

\newcommand{\ubar}[1]{{\underline{#1}}}
\def\sheafisom{{\stack{Isom}}}

\def\udot{^{\bullet}}
\def\ra{\rightarrow}
\def\tensor{\otimes}
\def\iso{\cong}
\def\cross{\times}
\def\inject{\hookrightarrow}
\def\bs{\backslash}
\def\sriso{\stackrel{\sim}{\rightarrow}}
\def\derham{{\rm dR}}
\def\cris{{\rm cris}}

\DeclareMathOperator{\spec}{Spec}
\DeclareMathOperator{\spf}{Spf}

\def\dual{^\vee}
\def\inv{^{-1}}
\def\twiddle{\sim}
\newcommand{\half}[1]{\frac{#1}{2}}
\newcommand{\powser}[1]{[\![#1]\!]}

\newenvironment{alphabetize}{\begin{enumerate}

}{\end{enumerate}}

\begin{document}

\title{Arithmetic occult period maps}

\author{Jeffrey D. Achter}
\email{j.achter@colostate.edu}
\address{Department of Mathematics, Colorado State University, Fort
Collins, CO 80523}
\urladdr{https://www.math.colostate.edu/~achter}


\begin{abstract}
%\xout{Nope, concrete.}

  Several natural complex configuration spaces admit surprising
  uniformizations as arithmetic ball quotients, by identifying each
  parametrized object with the periods of some auxiliary object.  In
  each case, the theory of canonical models of Shimura varieties gives
  the ball quotient the structure of a variety over the ring of
  integers of a cyclotomic field.  We show that the
  (transcendentally-defined) period map actually respects these
  algebraic structures, and thus that occult period maps are
  arithmetic.  As an intermediate tool, we develop an arithmetic
  theory of lattice-polarized K3 surfaces.
\end{abstract}

\maketitle

\section{Introduction}

It occasionally happens that  complex varieties of a
specified type are
parametrized by an arithmetic quotient of a unit ball in a surprising
way.  We situate this remark by recalling some aspects of
the primordial period map. Consider $\stack M_g$, the moduli space of
smooth projective curves of genus $g\ge 2$.  Given a smooth projective
curve of genus $g$, the possibilities for its period lattice
are naturally parametrized by the quotient of $\hh_g$, the Siegel
upper-half space of dimension $g(g+1)/2$, by $\SP_{2g}(\integ)$.  The
classical Torelli theorem asserts that the corresponding map $\tau_{g,\cx}: \stack M_g(\cx) \to \hh_g/\SP_{2g}(\integ)$ is an inclusion.  Even more is
true.  On one hand, $\stack M_g$ has a natural structure as a moduli
space over $\integ$.  On the other hand, let $\stack A_g$ be the moduli
space of principally polarized abelian varieties of dimension $g$; it,
too, is a space
over $\integ$.  Via the identification $\stack A_g(\cx) \iso
\hh_g/\SP_{2g}(\integ)$, we endow the latter with a structure over
$\integ$, as well.  The key arithmetic fact about the Torelli map is
that $\tau_{g,\cx}$, \emph{a priori} a transcendental map, descends to
a morphism $\tau_{g}: \stack M_g \inject \stack A_g$ over $\integ$
(e.g., \cite[\S 7.4]{mumfordGIT}).
Still, as soon as $g > 3$, $\dim \stack M_g < \dim \stack A_g$.  This means that many
of the arithmetic structures on $\stack A_g$, such as Hecke operators
and modular forms, don't readily make sense for the moduli space of
curves.

In the special case where $g=4$, however, we have the intriguing observation of
Kond\=o \cite{kondog4} that $\stack M_4(\cx)$ is very close to an
arithmetic quotient of $\ball^{9}$, the complex unit 9-dimensional
ball.  Slightly more precisely, let $\stack N_4$ denote the (open,
dense) locus of non-hyperelliptic curves.  Kond\=o shows that there
exist an arithmetic group $\Gamma$ of automorphisms of $\ball^{9}$
and an open immersion $\stack N_4(\cx) \inject \ball^{9}/\Gamma$.  (He
even characterizes the image.)  Instead of analyzing the periods of a
non-hyperelliptic curve $C$, the construction of \cite{kondog4}
proceeds by constructing an auxiliary variety $Z$ associated to $C$,
and analyzing \emph{its} periods.  Kudla and Rapoport \cite{kroccult}
-- who call such a period map \emph{occult}, in recognition of its hidden
nature -- observe that the theory of canonical models of Shimura
varieties produces a distinguished algebraic model of
$\ball^{9}/\Gamma$ over $\rat(\zeta_3)$.  They prove that Kond\=o's
occult period map actually respects the structures of $\stack N_4$ and
$\ball^{9}/\Gamma$ as varieties over $\rat(\zeta_3)$, and conjecture
that it extends to a map of integral canonical models over
$\integ[\zeta_3,1/3]$.  (They also note certain stack-theoretic
issues, which have since been resolved by Zheng \cite{zhengoccult};
see Remark \ref{rem:zheng} below.)

In fact, several different situations are known in which, for some
moduli space $\stack V$ of low-complexity varieties, an occult
period map yields an open immersion $\tau_{\stack V,\cx}:\stack V(\cx) \inject \ball^{\dim
  V}/\Gamma_{\stack V}$; see, for instance, \cite{dolgachevkondo} and
\cite{kroccult}, or even \S \ref{sec:examples} below, for examples.  In each case known
to the author, the theory of integral canonical models of Shimura
varieties provides a distinguished model of $\shim_{\Gamma_{\stack V}}(\ball^{\dim
  \stack V})$ of $\ball^{\dim \stack V}/{\Gamma_{\stack V}}$
over $\integ[\zeta_n,1/n]$ for some $n = n(\stack V)$.  The goal of the
present paper is to show that, in many cases, $\tau_{\stack V,\cx}$ descends
to a morphism $\stack V \inject \stack \shim_{\Gamma_{\stack V}}(\ball^{\dim \stack V})$ over
$\integ[\zeta_n,1/2n]$.

Many of the original constructions involve somehow building a K3
surface out of the original variety, and then analyzing the periods of
the corresponding K3 surface.  Consequently, much of the work of the
present paper is in analyzing moduli spaces $\stack R_{L,\ubar\chi}$ of
K3 surfaces polarized by the lattice $L$ and equipped with a suitable
action of $\mmu_n$.  A representative result -- the notation is
defined later in the paper -- is:

\begin{prop*}
There are morphisms of stacks over $\integ[\zeta_3,1/6]$:
\[
\xymatrix{
\mk^\circ_{L_4,\ubar\chi_4} \ar[r]^{\kappa_4} \ar[d]^{\tau_{L_4,\ubar\chi_4}} & \stack N_4 \\
\stack \shim^{(L_4,\ubar\chi_4)}&
}
\]
where $\kappa_4$ induces an isomorphism of coarse moduli spaces, and $\tau_{L_4,\ubar\chi_4}$ induces an open immersion $\mk^\circ_{L_4,\ubar\chi_4}(\cx) \inject \shim^{(L_4,\ubar\chi_4)}(\cx)$.
\end{prop*}
The statement over $\cx$ is, essentially, \cite[Thm.~1]{kondog4}; taking fibers over $\rat(\zeta_3)$
recovers the descent result \cite[Thm. 8.1]{kroccult}.

\subsection*{Acknowledgments}
My interest in this circle of ideas was sparked when
S.~Casalaina-Martin told me about the remarkable identification of the
moduli of complex cubic surfaces as a ball quotient
\cite{actsurfaces,dvgk}.  It's a pleasure to thank him for continued
inspiration and insights.  I thank M.~Rapoport, K.~Madapusi Pera
and the referee for helpful comments, and I.~Dolgachev for suggesting
that this paper should address the quasi-polarized situation.  This work was partly supported
by grants from the Simons Foundation (637075) and the National
Security Agency (H98230-14-1-0161, (15/16)-1-0247).

\section{Notation on lattices}
\label{sec:lattice}

\subsection{Lattices}
Throughout this paper, a {\em lattice} is a free $\integ$-module $L$ of
finite rank, equipped with a nondegenerate, symmetric bilinear
pairing $(\cdot,\cdot)$ (often notationally suppressed).  For any
nonzero $n$, we let  $L(n)$ denote the lattice
with the same underlying group structure as $L$ and with pairing
$(\cdot,\cdot)_{L(n)} = n (\cdot,\cdot)_L$.
We follow the conventions of \cite{huybrechtsk3} for lattices.  Lattices used here include:
\begin{itemize}
\item[$U$] the hyperbolic plane, which has rank
$2$ and pairing $\left(\begin{smallmatrix} 0&1\\1&0\end{smallmatrix}\right)
  $;
\item[$\ang 1$] the lattice of rank $1$ and pairing $(1)$;
\item[$E_8$] the unique positive definite unimodular lattice of rank
  $8$;
\item[$A_n$, $D_n$] the (positive) lattice associated to the Dynkin
  diagrams of type $A_n$ and $D_n$, respectively (in particular, $A_1
  \iso \ang 2$);
\item[$\lk$] the lattice $U^{\oplus 3}\oplus E_8(-1)^{\oplus 2}$, of signature $(3^+,19^-)$;
\item[$V$] the lattice of rank $2$ and pairing $\begin{pmatrix} 2&1\\1&-2 \end{pmatrix}$.
\end{itemize}

The pairing induces an inclusion $L\inject L\dual$; the discriminant
of $L$ is the finite abelian group $\disc(L) =
L\dual/L$, and we set $\Delta_L = [L\dual:L] = \#\disc(L)$.  Finally,
let
\[
d(L) = \gcd(\st{d \in {\mathbb N}: \exists \ang{2d}\inject L\text{
    primitive}}).
\]
For use in \S \ref{sec:examples}, we record the following elementary
facts:

\begin{lemma}
\begin{alphabetize}
\item If $M$ is a primitive sublattice of $L$, then $d(L)|d(M)$.
\item $d(U(n)) = n$, while $d(A_1 \oplus A_1(-1)^{\oplus 2}) =
  d(V\oplus A_4(-1)) = 1$.
\end{alphabetize}
\end{lemma}

\subsection{Orthogonal  groups}

To $L$ we associate the orthogonal group  $\orth_L$, with
connected component of identity the special orthogonal group
$\sorth_L$.  Since we start with an integral model for $\sorth$ as the
automorphism group of the lattice $L$, we have a natural definition of
$\sorth_L(R)$ for any ring $R$.  In particular, $\sorth_L(\integ_p)$
is well-defined and, by definition, a hyperspecial subgroup of
$\sorth_L(\rat_p)$.  (In Section \ref{SS:shimintegral}, a choice of
hyperspecial subgroup is necesary for the construction of a canonical
integral model of the relevant Shimura variety.  In fact, since
$\sorth_L$ is adjoint and split, $\sorth_L(\rat_p)$ admits a unique
$\sorth_L(\rat_p)$-conjugacy class of hyperspecial subgroups.)

We will often have cause to work with a lattice $L$ which comes
equipped with a primitive embedding $\iota:L \inject \lk$.  With a
slight abuse of notation, we will write $L^\perp$ for the orthogonal
complement $\iota(L)^\perp$ of $\iota(L)$ in $\lk$.  Set $\orth^L = \orth_{L^\perp}$ and $\sorth^L =
\sorth_{L^\perp}$.
An element of
$\orth^L(\integ)$ extends to an element of $\orth_\lk(\integ)$ acting
trivially on $L$ if and only if it
acts trivially on $\disc(L)$ (e.g., \cite[14.2.6]{huybrechtsk3}).
More generally, if $R$ is flat over $\integ$, then an element of
$\orth^L(R)$ extends to an element of $\orth_L(R)$ acting trivially on
$L\tensor R$ if and only if it
acts trivially on $\disc(L)\tensor R$ \cite[Lemma 2.6]{peraspin}.  The subgroup of
admissible orthogonal automorphisms of $L^\perp$ is the group scheme
$\til \orth ^L$ fitting in the short exact sequence
\[
\xymatrix{
1 \ar[r] & \til \orth^L \ar[r] & \orth^L \ar[r] & \aut(\disc(L))
\ar[r] & 1;
}
\]
on points, we have
\begin{align*}
\til\orth^L(R) &= \st{ g\rest{L^\perp\tensor R} : g \in \orth_\lk(R)}\\
&= \ker \left(\orth^L(R) \ra \aut(\disc(L))(R) \right).
\end{align*}
If $\til g \in \til\orth^L(R)$, then there is a (necessarily unique) $g
\in \orth_\lk(R)$ such that $g \rest{L^\perp} = \til g$ and $g
\rest{L} = \id_L$.  In this way, $\til \orth^L(R)$ is naturally
identified with a subgroup of $\orth_\lk(R)$.
The group scheme $\til\sorth^L := \til \orth^L \cross_{\orth^L}
\sorth^L$ represents admissible automorphisms of determinant one.

\section{Families of K3 surfaces}

\subsection{K3 surfaces}
\label{subsec:k3def}

Let $k$ be an algebraically closed field.  A K3 surface over $k$ is a
smooth, complete irreducible surface $Z/k$ with trivial canonical
bundle $\omega_Z := \Omega^2_{Z/k} \iso \calo_Z$ and such that
$H^1(Z,\calo_Z) = 0$.  Like any smooth complete surface, a K3 surface
is projective, and its middle Hodge numbers are $(h^{2,0},h^{1,1},
h^{0,2}) = (1,20,1)$.
If $k = \cx$, then the Betti cohomology group $H^2(Z,\integ)$, endowed
with the intersection pairing $(\cdot.\cdot)$, is isomorphic to $\lk$.
The natural pairing
\[\xymatrix{
 \Omega_Z \cross \Omega_Z \ar[r] & \omega_Z
}
\]
induces an isomorphism of $\calt_Z \iso \Omega_Z$.  Since
$H^1(Z,\calo_Z)$ is trivial so is $\pic^0(Z)$, and the N\'eron-Severi
group $\ns(Z)$ coincides with the Picard group $\pic(Z)$.  As for any
surface, $\ns(Z)$ is a free, finitely generated $\integ$-module,
equipped with a symmetric, nondegenerate pairing
\[
\xymatrix{
\ns(Z) \cross \ns(Z) \ar[r]^-{(\cdot.\cdot)} & \integ.
}
\]
The
intersection pairing  is even, nondegenerate, and of signature
$(1,\rank(\ns(Z))-1)$.

Following Rizov \cite{rizovmoduli} and successors, we say that a a
relative K3 surface, or K3 space, over a scheme $S$ is an algebraic
space $Z \ra S$ such that each geometric fiber is a K3 surface.
If
$Z \ra S$ is a relative K3 surface, then $\calh^2_\derham(Z/S)$ and
$\calh^{2,0}(Z/S)$ are locally free sheaves on $S$ \cite[Prop.\
2.2]{delignelift}, \cite[\S 3.4]{perak3} of respective ranks $22$ and $1$.

\subsection{Categories of K3 surfaces}
\label{subseccatk3}
We will study three different sorts of moduli spaces of K3 surfaces.

Classically, one has $\mk_{2d}^\circ$, the category of K3 surfaces
equipped with a primitive ample polarization of degree $2d$.  On points,
$\mk_{2d}^\circ(S)$ is the category of pairs $(Z \ra S, \lambda)$, where
$Z \ra S$ is a K3 space, and $\lambda \in \pic_{Z/S}(S)$ is
\'etale-locally represented by an ample line bundle of
self-intersection degree $2d$ which is not a nontrivial tensor power
of any other line bundle. This is a subcategory of $\mk_{2d}$, the
category of K3 surfaces equipped with a primitive quasi-ample
polarization of degree $2d$.  (A quasi-ample, or pseudo-ample,
polarization is \'etale-locally a line bundle which is big and nef.)

Choose a generator $e_0$ for $\ang{2d}$.  To specify data $(Z \to S,
\lambda)$ is to specify a K3 space $Z \to S$ and an embedding of
lattices $\ang{2d} \inject \pic_{Z/S}(S)$ which takes $e_0$ to the
class of an ample line bundle.  (Recall that if $Z/k$ is a K3 surface
over an algebraically closed field, and if $v \in \pic_{Z/k}(k)$
satisfies $(v,v)>0$, then exactly one of $v$ and $-v$ and represents
the class of an ample line bundle.  The choice of a
generator for $\ang{2d}$ is equivalent to the choice of a ``positive
cone'' in $\ang{2d}\tensor \real \iso \real$.)

More generally, we consider lattice-polarized K3 surfaces.
Let $L$ be a primitive sublattice of $\lk$ of signature $(1,r-1)$.  The set
\[
\st{v\in L\tensor \real: (v,v) >0}
\]
has two connected components. Choose one such, $V^+$, and, as in \cite[\S 1]{dolgachev_lattice} or \cite[\S 10]{dolgachevkondo}, define an abstract "ample cone" $\mathcal C(V^+)$, an open subset of $V^+$, and let $L^+ = L\cap \mathcal C(V)^+$.
We suppress the choice of $V^+$ (and, thus, $L^+$) from the notation,
and let $\mk_L^\circ$ be the category of ample $L$-polarized K3 surfaces.
Objects in $\mk_L^\circ$ are isomorphism classes of pairs
$(Z \ra S, \alpha)$, where $Z \ra S$ is a K3 space and
$\alpha: L \inject \pic_{Z/S}(S)$ is a primitive embedding of
lattices such that $\alpha(L^+)$ contains the class of an ample line
bundle.  (Since $\alpha$ is a primitive embedding, it is equivalent to
ask that $\alpha(\mathcal C(V^+))$ contains such a class.)
  Here, we declare
that two such data $(Z_i \ra S, \alpha_i)$ are isomorphic if there is
an isomorphism $f:Z_1 \ra Z_2$ such that
$f^*\alpha_2 = \alpha_1$.  We define $\mk_L$, the category of
$L$-polarized K3 surfaces, analogously except that it is only assumed
that $\alpha(L^+)$ contains the class of a quasi-ample line bundle.

(Of course, one can also make the
definition of an $L$-polarized K3 surface without keeping track of a
positive cone, provided one is willing to identify $\alpha$ and
$-\alpha$.  For a Shimura-theoretic justification for this approach,
see Remark \ref{rem:taelmantwist} and, ultimately, \cite[\S
5]{taelman17}. Consequently, the choice of $L^+$ is suppressed here.)

For a finite group scheme $G$, let $\mk_{L,G}^*$ be the category of tuples
$(Z \ra S, \alpha, \rho)$ where $(Z\ra S, \alpha) \in \mk_L(S)$ and
$\rho: G_S \inject \aut_S(Z\ra S,\alpha)$ is a monomorphism of group
schemes.  If $\#G$ is invertible on $S$ -- equivalently, if the
cardinality of $G$ is relatively prime to the characteristic exponent
of all residue fields of points of $S$ -- then representations of $G$
on $\calo_S$-modules are rigid, and thus $\calh^2_\derham(Z/S)$ and
$\calh^{2,0}(Z/S)$ are locally free sheaves of
$\calo_S[G]$-modules.

We now specialize to the case $G = \mmu_n$, and restrict to the
category of schemes over $\integ[1/2\Delta_Ln]$.  Let $\chi^\omega$ be a faithful character of $\mmu_n$, $\chi_0$ the
trivial character, and $\chi$
be an arbitrary character; let $m(\chi^\omega) = m_\chi(\chi^\omega)$
and $m(\chi_0) = m_\chi(\chi_0)$ be the multiplicities of,
respectively, $\chi_0$ and $\chi^\omega$ in $\chi$.
Let $\mk_{L,\mmu_n,\chi^\omega, \chi}$
be the open and closed substack of $\mk_{L,\mmu_n}^*$ parametrizing
those $(Z \ra S, \alpha, \rho)$ such that
\begin{itemize}
\item $\mmu_n$ acts on $\calh^{2,0}(Z/S)$ via $\chi^\omega$;
\item $\mmu_n$ acts on $\calh^2(Z/S)$ via $\chi$; and
\item $m_\chi(\chi_0) = \rank(L)$.
\end{itemize}
In particular, the action of $\mmu_n$ is purely non-symplectic, in the
sense that no nontrivial section of $\mmu_n$ fixes a nonzero
holomorphic 2-form.

Suppose $S$ is irreducible and $\bar s$ is a geometric point of $S$.
Because $2n$ is invertible on $S$, representations of $\mmu_n$ on
$\calo_S$-modules are rigid.  In particular,
 the character of the action of
$\mmu_n$ on $\calh^2(Z/S)$ is determined by the action on
$\calh^2_\derham(Z_{\bar s})$.  Moreover, it is equivalent to specify
this character in terms of the action of $\mmu_n$ on
$H^2_\cris(Z_{\bar s})$, or any of the \'etale cohomology groups
$H^2(Z_{\bar s},\rat_\ell)$ \cite[Thm~2.2]{katzmessing}, or (since K3
surfaces have torsion-free cohomology) $H^2(Z_{\bar s},\integ_\ell)$.

We will often use the symbol $\ubar\chi$ to denote the collection of
data $(\mmu_n,\chi^\omega,\chi)$, and thus write $\mk_{L,\ubar\chi}$
for $\mk_{L,\mmu_n,\chi^\omega,\chi}$, etc.

%\subsection{Interlude}

The possibilities for data $(L,\ubar\chi)$ such that $\mk_{L,\ubar\chi}(\cx)$ is nonempty are reasonably well-understood.
On one hand, it is not hard to see that a purely
non-symplectic group of automorphisms is finite and cyclic; the
possible orders of such a group are also known
\cite{machidaoguiso98}.
On the other hand, starting with the work of Nikulin, one has a good
classification of primitive sublattices of $\lk$ \cite{nikulin79}.  In
\S \ref{sec:examples} we will see a number of explicit examples of
naturally occurring families of lattice-polarized K3 surfaces.

\subsection{Stacks of K3 surfaces}
\label{subsecstackk3}

It turns out that each $\mk_{2d}$, $\mk_L$, and $\mk_{L,\ubar\chi}$ is a Deligne-Mumford stack.  Indeed, Rizov proves
that $\mk_{2d}^\circ$ is a Deligne-Mumford stack \cite[Thm.\ 4.3.3]{rizovmoduli}, and Beauville essentially proves the same of $\mk_L$ in
\cite{beauvillefanok3}.  The partial compactification $\mk_{2d}$ of
$\mk_{2d}^\circ$ is also Deligne-Mumford and even smooth over
$\integ[1/2d]$ \cite[Prop.~2.1]{maulik_k314},
albeit no longer separated (e.g., \cite[5.1.4]{huybrechtsk3}).
Rather than working {\em ab ovo} to study $\mk_L$ and $\mk_{L,G}^*$,
we find it expedient to bootstrap
from Rizov's work.

It is convenient to make at the outset a few (arbitrary) choices; the
final claims are intrinsic, and independent of these choices.  Let
$e_1, \cdots, e_r$ be a $\integ$-basis for $L$.  Fix some $\lambda \in
L^{+}$.% such that $2d(\lambda) := (\lambda . \lambda) > 0$ and,
%for each $i$, $\lambda+e_i \in L^{+}$.

\begin{lemma}
  \label{lemmkstack}
The category $\mk_L$ is a stack over $\spec \integ$.
\end{lemma}

\begin{proof}
  We must show that the diagonal $\mk_L \ra \mk_L  \cross \mk_L$ is representable, and that \'etale descent in the
  category $\mk_L$ is effective.

  For the first claim, it suffices to show that if $(Z_1 \ra S,
  \alpha_1)$ and $(Z_2 \ra S, \alpha_2)$ are elements of $\mk_L(S)$, then
\[
\sheafisom((Z_1,\alpha_1), (Z_2,\alpha_2))
\]
is representable by a scheme over $S$.  The functor $\sheafisom(Z_1,
Z_2)$ is represented by a scheme over $S$.  (In fact, it is open in
$\operatorname{Hilb}(Z_1\cross Z_2)$.)  Pullback by isomorphisms gives a pairing $\sheafisom(Z_1,Z_2) \times_S \pic_{Z_2/S} \to \pic_{Z_1/S}$.  Consider some $i$ between $1$ and $r$.  Pulling back the pairing by the section $\alpha_2(e_i):S \to \pic_{Z_2/S}$ induces a morphism $\beta_i: \sheafisom(Z_1,Z_2) \to \pic_{Z_1/S}$.
Then
\[
\sheafisom((Z_1,\alpha_1(e_i)), (Z_2,\alpha_2(e_i))) := \sheafisom(Z_1,Z_2)\times_{\pic_{Z_1/S}, \alpha_1(e_i)} S
\]
is the sub-scheme of $\sheafisom(Z_1,Z_2)$ parametrizing those isomorphisms which take $\alpha_2(e_i)$ to $\alpha_1(e_i)$.  Insofar as $ \sheafisom((Z_1,\alpha_1),(Z_2,\alpha_2))$ is the fiber product over $\sheafisom(Z_1,Z_2)$ of the $r$ different schemes $\sheafisom((Z_1,\alpha_1(e_i)),(Z_2,\alpha_2(e_i)))$, it too is represented by a scheme.

For the second, let $T \ra S$ be \'etale and let $(\til Z \ra T, \til
\alpha)\in \mk_L(T)$ be equipped with $T/S$ descent data.  In
\cite[Lemma 4.3.7]{rizovmoduli}, the author shows that in the ample case $(\til Z \ra T,
\alpha(\lambda))$ descends, as a polarized K3 space, to $S$; the quasi-polarized case follows from \cite[\S 2]{maulik_k314}.  Since
$\pic_{Z/S}$ is a sheaf in the \'etale topology, each $\til
\alpha(e_i)$ descends to $Z/S$.
\end{proof}

\begin{lemma}
\label{lemMLDM}
The category $\mk_L$ is a Deligne-Mumford stack over $\spec\integ$.
\end{lemma}

\begin{proof}
  Because $\mk_{2d(\lambda)}$ is known to be a Deligne-Mumford
  stack, it suffices to show that the forgetful morphism
\[\xymatrix{
\mk_L \ar[r]^{\phi_\lambda} & \mk_{2d(\lambda)} \\
(Z \ra S, \alpha) \ar@{|->}[r] & (Z \ra S, \alpha(\lambda))
}
\]
is relatively representable \cite[Prop. 4.5.(ii)]{lmbstacks}.  Now proceed as in \cite{beauvillefanok3}.  Let $H_\lambda \subset \orth_L(\integ)$ be the subgroup which stabilizes $\lambda$.  Since $\lambda^\perp$ is negative definite, $H_\lambda$ is finite.   Given an $S$-point $S \ra \mk_{2d(\lambda)}$,
\[
\mk_L \cross_{\phi_\lambda, \mk_{2d(\lambda)}}S,
\]
if nonempty, is a torsor under $H_\lambda$, and in particular is representable.
\end{proof}

\begin{proposition}
\label{propmlsmooth}
The category $\mk_L$ is a smooth Deligne-Mumford stack over $\spec
\integ[1/2\Delta_L]$
of relative dimension $20-r$.
\end{proposition}

\begin{proof}
  Given Lemma \ref{lemMLDM}, it suffices to show that the local
  deformation space of a  quasi-ample $L$-polarized K3 surface is smooth.
  In characteristic zero, this is asserted in
  \cite[Prop.~2.1]{dolgachev_lattice}, and details are provided in \cite[Prop.~1.4]{beauvillefanok3}.   In positive characteristic, this follows from Deligne
  and Illusie's deformation theory; see Proposition \ref{propdeformlattice}
  below.
\end{proof}

Now consider the moduli spaces of lattice-polarized K3
surfaces with group action.

\begin{lemma}
\label{lemautproper}
Suppose $(Z \ra S, \alpha) \in \mk_L(S)$.  Then $\aut(Z\ra S, \alpha)$ is represented by a proper finite group scheme over $S$.
\end{lemma}

\begin{proof}
The automorphism
functor $\aut_S(Z)$ is represented by a separated, unramified group
scheme over $S$
\cite[Thm.~3.3.1]{rizovmoduli}.  Moreover, $\aut_S(Z\ra S,
\alpha(\lambda))$ is a closed, finite, subgroup scheme of
$\aut_S(Z)$ (see \cite[Prop.~3.3.3]{rizovmoduli} for the polarized case; the extension to quasi-polarizations follows from \cite[p.2369]{maulik_k314}).
As in the proof of Lemma \ref{lemmkstack}, $\aut_S(Z\ra S,\alpha)$ is a sub-$S$-group scheme of $\aut_S(Z\ra S, \alpha(\lambda))$.
The claimed properness follows from \cite[Thm.~2]{matsusakamumford}.
\end{proof}

\begin{lemma}
\label{lemforgetGclosed}
The forgetful morphism $\mk_{L,G}^* \ra \mk_L$ is finite, and $\mk_{L,G}^*$ is a Deligne-Mumford stack.
\end{lemma}

\begin{proof}
First, the forgetful functor $\mk_{L,G}^* \ra \mk_L$ is
relatively representable.  Indeed, for any affine scheme $S$ and any
$(Z \ra S,\alpha) \in \mk_L(S)$, both $G_S$ and $\aut_S(Z\ra
S,\alpha)$ are relatively representable, and thus $\hom(G_S,
\aut_S(Z\ra S,\alpha))$ is representable, too; and the condition that
a homomorphism be injective is open.  Therefore, $\mk_{L,G}^*$ is
also a Deligne-Mumford stack.  The properness (and finitude) in Lemma
\ref{lemautproper} imply that $\mk_{L,G}^* \ra \mk_L$ is
proper and quasifinite, thus finite.
\end{proof}

 \begin{proposition}
 The category $\mk_{L,\ubar\chi}$ is a smooth Deligne-Mumford stack
 over $\integ[\zeta_n,1/6\Delta_Ln]$
of relative dimension $m(\chi^\omega)-1$.
 \end{proposition}

 \begin{proof}
 All that needs to be checked is smoothness; this is done in Lemma \ref{lemdeformaut}.
 \end{proof}

\subsection{Local calculations}
\label{subseclocal}

If $Z/\cx$ is a complex K3 surface, then (the local Torelli theorem
asserts that) the deformation theory of $Z$ is well-captured by its
Hodge theory.  In particular, let $\defo(Z)$ be the deformation
functor of $Z$, with base point $s$.
Then there is a canonical isomorphism
\[
\xymatrix{
T_s\defo(Z) \ar[r]^-\twiddle &
\hom(H^{2,0}(Z),H^{2,0}(Z)^\perp/H^{2,0}(Z)).
}
\]
There is a parallel deformation theory for K3 surfaces in arbitrary
characteristic, which we review here.  Let $k$ be an algebraically
closed field of characteristic $p >0 $, with ring of Witt vectors $W = W(k)$.

Let $Z/k$ be a K3 surface.  The deformation functor $\defo(Z)$ is formally smooth over $\spf W$ of relative dimension 20; $\defo(Z)$ is pro-represented by a formal scheme noncanonically isomorphic to $\spf W\powser{t_1, \cdots, t_{20}}$ \cite[Cor.\ 1.2]{delignelift},
\cite[4.1.1]{rizovmoduli}.  Let $s$ be the base point of $\defo(Z)$, corresponding to $Z/k$ itself.

\begin{lemma}
\label{lemtangent}
There is a canonical isomorphism of $k$-vector spaces
\[\xymatrix{
\ts_s \defo(Z) \ar[r]^-\twiddle & \hom(\fil^2 H^2_\derham(Z/k), \gr^1
H^2_\derham(Z/k)).
}
\]
\end{lemma}

\begin{proof}[Sketch] See \cite[2.4]{delignelift} \cite[5.2]{nygaardogus}, \cite[5.1]{ogusicm}.  Briefly, let $A$ be a
  nilpotent extension of $k$ with a divided power structure.  The
  intersection pairing $(\cdot.\cdot)$ extends to a pairing on the
  crystal $H^2_\cris(Z)$. To give a  deformation of $Z$ to $A$ is to
  lift $\fil^2 H^2_\cris(Z)(k)$ to an isotropic  direct summand of
  $H^2_\cris(Z)(A)$.  Now use the fact \cite[(2.3.7)]{delignelift} that the
  orthogonal complement to $\fil^2 H^2_\cris(Z)(k)$ is $\fil^1 H^2_\cris(Z)(k)$.
\end{proof}

If $2d$ is invertible in $k$ and $(Z,\lambda) \in \mk_{2d}(k)$, then $\defo(Z,\lambda) \subset \defo(Z)$
is prorepresentable, and formally smooth of dimension $19$ over $\spf
W$ \cite[1.5 and 1.6]{delignelift}, \cite[3.8]{perak3}, \cite[4.1.3]{rizovmoduli}.  More generally, we have:

\begin{proposition}
\label{propdeformlattice}
Let $L$ be a lattice of rank $r$ and  discriminant $\Delta_L$, and suppose that $\Delta_L$ is invertible in $k$.  Let $(Z/k,\alpha)\in \mk_L(k)$ be an $L$-polarized K3 surface.  Then $\defo(Z,\alpha)$
is prorepresentable and formally smooth of dimension $20 - r$
over $\spf W$.
\end{proposition}

\begin{proof}
As in \ref{subsecstackk3},
let $e_1, \cdots, e_r$ be a $\integ$-basis for $L$, and let $\call_i =
\alpha(e_i) \in \pic(Z)$.  Then $\defo(Z,\alpha) = \defo(Z,\st{\call_1,
  \cdots, \call_r})$ is the largest formal subscheme of $\defo(Z)$ to
which each of the line bundles $\call_i$ extends.  Thus,
$\defo(Z,\alpha)$ is the scheme theoretic intersection of the
$\defo(Z,\call_i)$.  Now, each $\defo(Z,\call_i)$ is the vanishing
locus in $\defo(Z)$ of a single function $f_i$ \cite[1.5]{delignelift}.  So
(any component of) $\defo(Z,\alpha)$ has codimension at most $r$ in
$\defo(Z,\alpha)$, and it suffices to show that the dimension of the
tangent space of $\defo(Z,\alpha)$ at the base point is exactly
$20-r$.

Since $Z$ is smooth and proper, there is a crystalline Chern class map
$c_1: \ns(Z) \ra H^2_\cris(Z/W)$.  Moreover, since the crystalline
cohomology of $Z$ is torsion-free and the Hodge to deRham spectral
sequence for $H\udot(Z/W)$ degenerates at $E_1$, the Chern class map
yields an inclusion $\bar c_1:\ns(Z)/p\ns(Z) \inject H^2_\cris(Z/k)
\iso H^2_\derham(Z/k)$ \cite[Rem.\ 3.5]{delignelift}.

Let $\bar c_1(\alpha) \subset H^ 2_\derham(Z/k)$ be the span of
$\bar c_1(\alpha(e_1)), \cdots, \bar c_1(\alpha(e_r))$; it is actually
a subspace of $\fil^1 H^ 2_\derham(Z/k)$ (e.g., \cite[3.4]{perak3}).

Since $\Delta_L$ is invertible in $k$, $\call_1, \cdots, \call_r$ are
linearly independent in $\ns(Z)/p\ns(Z)$, and thus $\dim \bar c_1(\alpha) =r$.
Now use the fact (e.g., \cite[Thm.\ 3.8(3)]{perak3}, modelled
after \cite[5.1.2]{ogusicm}) that a line bundle $\call$
extends to a given deformation $\til Z/A$ if and only if $c_1(\call)$
is orthogonal to the corresponding lift $\fil^2 H^2_\cris(\til
Z)(A)$.
Under the isomorphism of Lemma \ref{lemtangent}, we see that
the tangent space $\ts_s\defo(Z,\alpha)$ corresponds to homomorphisms
from $\fil^2 H^2_\derham(Z/k)$ into the orthogonal complement of
$\bar c_1(\call_1), \cdots, \bar c_1(\call_r)$ in
$\gr^1H^2_\derham(Z/k)$.  Because these Chern classes are linearly
independent over $k$ and $(\cdot.\cdot)$ is nondegenerate, the
codimension of $\ts_s\defo(Z,\alpha)$ in $T_s\defo(Z)$ is $r$.
\end{proof}

We now suppose that data $\ubar\chi = (\mmu_n,\chi^\omega,\chi)$ is chosen so that $\mk_{L,\ubar\chi}$ is nonempty, and further assume that $n$ is invertible in $k$.

\begin{lemma}
\label{lemdeformaut}
Suppose $(Z, \alpha, \rho) \in \mk_{L,\ubar\chi}(k)$ and
$\operatorname{char}(k)\nmid 2\Delta_Ln$.  The (equicharacteristic) tangent space to $\mk_{L,\ubar\chi}$ at $(Z,\alpha,\rho)$ has dimension $m(\chi^\omega)-1$.
\end{lemma}

\begin{proof}
By the crystalline local Torelli theorem
(\cite[Rem.~3.23]{berthelotogusI} and \cite[Lemma 3.1]{jang17}; see
also \cite[Thm.~5.3 and
Rem.~(5.3.1)]{nygaardogus} and \cite[5.1.2]{ogusicm} for
characteristic at least $5$), it suffices to
identify the sublocus of $\defo(Z,\alpha)$ to which the $G$-action on
$H^2_\cris(Z)$ extends.
Thus,  let $\bar c_1(\alpha)^\perp$ be the orthogonal complement of $\bar c_1(\alpha) $, and consider the inclusions of formal deformation spaces $\defo(Z,\alpha,\rho)\subset
\defo(Z,\alpha) \subset \defo(Z)$.  Computing equicharacteristic tangent spaces at the
base point $s$, we have
\[
\xymatrix{
\ts_s \defo(Z) & \hom(\fil^2 H^2_\derham(Z/k),\gr^1
H^2_\derham(Z/k)) \ar[l]^-\twiddle \\
\ts_s \defo(Z,\alpha) \ar@{^{(}->}[u] & \hom( \fil^2 H^2_\derham(Z/k),
c_1(\alpha)^\perp/\fil^2H^2_\derham(Z/k))\ar[l]^-\twiddle \ar@{^{(}->}[u]\\
\ts_s \defo(Z,\alpha,\rho) \ar@{^{(}->}[u] & \hom_G( \fil^2 H^2_\derham(Z/k),
c_1(\alpha)^\perp/\fil^2 H^2_\derham(Z/k)) \ar[l]^-\twiddle \ar@{^{(}->}[u]\\
}
\]
By definition of $\mk_{L,\ubar\chi}$, the $\chi$-eigenspace of $\fil^1 H^2_\derham(Z/k)$ is fully contained in $c_1(\alpha)^\perp$; the result now follows.
\end{proof}

\section{Period maps for complex K3 surfaces}
\label{seccxperiod}

\subsection{Period maps}

The global complex Torelli theorem for K3 surfaces asserts that
the isomorphism class of a K3 surface $Z/\cx$ is determined by the
isomorphism class of $H^2(Z, \integ)$ as a polarized Hodge structure.
Via Hodge theory, one thus obtains a good global understanding of the
moduli of complex K3 surfaces, as follows.

Let $Z$ be a marked K3 surface, i.e., a K3 surface $Z/\cx$
equipped with an isometry $\phi: H^2(Z,\integ)
\sriso \lk$.  The image of $\phi_\cx(H^{2,0}(Z))$ determines an
element of the period domain
\[
\xx_{\lk} = \st{[\sigma] \in \proj(\lk \tensor \cx):
  (\sigma,\sigma) = 0, (\sigma,\bar\sigma) >0},
\]
and the isomorphism class of $Z$ is determined by the class of
$\phi_\cx(H^{2,0}(Z))$ in $\orth(\lk) \bs \xx_{\lk}$.  (We recall a
useful description of $\xx_{\lk}$ below in \ref{subsecperioddomain}.)
We
remind the reader that the action of the orthogonal group $\orth(\lk)$
on $\xx_{\lk}$ is not properly discontinuous, and thus the quotient
space $\orth(\lk)\bs \xx_{\lk}$ is not even Hausdorff.

Now suppose that $Z$ is equipped with a polarization $\lambda$ of
degree $2d$.  Recall that we have fixed an embedding $\iota:\l2d
\inject \lk$.  A marking $\phi$ of $Z$ induces an identification of
the primitive cohomology $P^2_\lambda(Z,\integ)$ with $\l2d^\perp
\subset \lk$, and thus $\phi_\cx(H^{2,0}(Z))$ lies in
\[
\xx^{\l2d} := \xx_{\l2d^\perp} = \st{[\sigma] \in \proj(\l2d^\perp \tensor \cx):
  (\sigma,\sigma) = 0, (\sigma,\bar\sigma) >0} \subset \xx_{\lk}.
\]

Recall (Section \ref{sec:lattice}) that $\til \orth^{\l2d}(\integ)$ consists
of those orthogonal automorphisms of $\l2d^\perp$ which admit an
extension to $\lk$ fixing $\ang{2d}$.  We thus have
a natural inclusion
\[
\xymatrix{
\til \orth^{\l2d}(\integ) \backslash
\xx^{\l2d} \ar@{^{(}->}[r] & \orth_\lk(\integ) \backslash \xx_{\lk}.
}\]
The strong Torelli theorem for polarized K3 surfaces implies that
there is an open immersion
\[
\xymatrix{
\mk^\circ_{2d,\cx} \ar@{^{(}->}[r]^-{\tau_{2d,\cx}} & \til \orth^{\l2d}(\integ) \backslash
 \xx^{\l2d}.
}
\]
 e.g., \cite[Thm.~6.3.4]{huybrechtsk3} which extends to an isomorphism
 of coarse moduli spaces
 $\mk_{2d,\cx} \to \til\orth^{\l2d}(\integ)\backslash \xx^{\l2d}$ \cite[Rem.~6.4.5]{huybrechtsk3}.

More generally, for a primitive sublattice $L\subset \lk$ of signature
$(1,r-1)$, we set
\begin{align*}
\xx^L = \xx_{L^\perp} &= \st{ [\sigma] \in
  \proj(L^\perp\tensor\cx) : (\sigma,\sigma) = 0,
  (\sigma,\bar\sigma)>0} \\
\end{align*}
and obtain an open immersion
\[
\xymatrix{
\mk_{L,\cx} \ar@{^{(}->}[r]^-{\tau_{L,\cx}} & \til \orth^L(\integ) \backslash
\xx^L;
}
\]
see \cite[\S 3]{dolgachev_lattice} and \cite[\S 11]{dolgachevkondo} for more details.

Finally, fix data $\ubar\chi = (\mmu_n,\chi^\omega,\chi)$ and suppose
that $(Z,\alpha,\rho) \in \mk_{L,\ubar\chi}(\cx)$.  A choice of
marking $\phi$ on $Z$ induces an action of
$\mmu_n$ on $\lk$ with character $\chi$.  The period point
$\phi_\cx(H^{2,0}(Z))$ then lies in
\begin{align*}
\xx^{L,\ubar\chi} &= \st{ [\sigma] \in
  \proj((L^\perp_\cx)(\chi^\omega)): (\sigma,\sigma) = 0, (\sigma,
  \bar\sigma)>0}
\\
\intertext{where we single out an eigenspace for the action of
  $\mmu_n$ on $L^\perp$ by}
L^\perp_\cx(\chi^\omega)&= \st{ v \in L^\perp\tensor\cx:
  \forall \zeta\in \mmu_n(\cx), \zeta v = \chi^\omega(\zeta)v}.
\end{align*}
Let $\orth^{L,\ubar\chi}$ be the group of automorphisms of $L^\perp$
which commute with the action of $\mmu_n$; if $R$ is a ring over
$\integ[\zeta_n]$, then
\[
\orth^{L,\ubar\chi}(R) = \st{ g \in \orth^L(R) : \forall \zeta \in
  \mmu_n(R), g\zeta v = \zeta g
  v}.
\]
Let $\til \orth^{L,\ubar\chi}$ be the subgroup of admissible
automorphisms of $L^\perp$.
Then we again have an open immersion
\[
\xymatrix{
\mk_{L,\ubar\chi,\cx}  \ar@{^{(}->}[r] &
\til\orth^{L,\ubar\chi}(\integ) \backslash \xx^{L,\ubar \chi}.
}
\]

Recall that $m_{\ubar\chi}(\chi^\omega) = m(\chi^\omega) = \dim L^\perp_\cx(\chi^\omega)$
is the multiplicity of the faithful character $\chi^\omega$ in the
representation $\chi$, and that $L \subset \lk$ is the module of
$\mmu_n$-invariants.

If $n \ge 3$, then $\chi^\omega$ is imaginary, and
\[
\xx^{L,\ubar\chi} \iso \ball^{m_{\ubar\chi}(\chi^\omega)-1},
\]
the complex unit ball of dimension $m_{\ubar\chi}(\chi^\omega)-1$.

If $n = 2$, then $\chi^\omega$ is real; $L^\perp_\cx(\chi^\omega) =
L^\perp_\real(\chi^\omega)\tensor\cx$; and $\xx^{L,\ubar\chi}$
is a type IV Hermitian symmetric space of dimension
$m_{\ubar\chi}(\chi^\omega)-1$.

\subsection{Period spaces}
\label{subsecperioddomain}

Since $L$ has signature $(1,r-1)$, $L^\perp$ has signature
$(2,19-(r-1))$.  It is traditional in the K3 literature to describe
the relevant period space as
\begin{align*}
\xx^L = \xx_{L^\perp} &\iso \frac{\orth^L(\real)}{\sorth_2(\real)\times
  \orth_{20-r}(\real) }
\intertext{To facilitate comparison with the Shimura variety
  literature, we prefer to recall that the special orthogonal group
  $\sorth^L(\real)$ already acts transitively on $\xx^L$, and we in fact
  have}
\xx^L  &\iso \frac{\sorth^L(\real)}{\sorth_2(\real)\times \sorth_{20-r}(\real)}.
\end{align*}
It is perhaps worth noting that the special orthogonal group of a
definite form is connected, while $\sorth_{2,20-r}(\real)$ has two
topological components, indexed by the two classes of the spinor
norm.  In particular, $\xx^L$ consists of two connected components,
say $\xx^{L+}$ and $\xx^{L-}$;
these components are stabilized by the component $\sorth^L(\real)^+
\subset \sorth^L(\real)$ of elements with trivial spinor norm.

Let $\Gamma \subset \orth^L(\real)$ be any arithmetic group.  Then
$\Gamma$ has finite covolume, and in particular meets every
topological component of $\orth^L(\real)$.  We have isomorphisms of
complex analytic spaces
\[
\Gamma \backslash\xx^L \iso (\Gamma \cap \orth^L(\real)^+)
\backslash\xx^{L+} \iso (\Gamma \cap \sorth^L(\real))\backslash
\xx^{L}.
\]
In particular, the period map is an open immersion
\[
\xymatrix{
\mk_{L,\cx} \ar@{^{(}->}[r]^-{\tau_L} &
\til\sorth^L(\integ)\backslash \xx^L.
}
\]

\section{Shimura varieties}

\subsection{Integral canonical models}
\label{SS:shimintegral}

We review some basic concepts concerning Shimura varieties, referring
the reader to \cite{deligne_travaux} for foundational material,
\cite{deligne_corvallis} for canonical models, \cite{kisin_intcan} for
integral canonical models, and \cite{taelman17} for stack-theoretic
issues.  All Shimura data are assumed to be of abelian type, so that
the cited references suffice.

Let $(G,\xx)$ be a Shimura datum, consisting of a reductive group
$G/\rat$ and a conjugacy class $\xx$ of homomorphisms
$\res_{\cx/\real}\gp_m \ra G_\real$ of $\real$-groups, subject to the
usual axioms. Further assume that $(G,\xx)$ is of abelian type.

Let $\kk \subset G(\aff_f)$ be a neat compact open subgroup of the finite
adelic points.  The holomorphic analytic quotient stack
$\shim^{\mathrm{an}}_\kk[G,\xx] := [G(\rat)\backslash(\xx\times G(\aff_f)/\kk)]$
is represented by the analytification  of a
smooth complex quasiprojective variety $\shim_\kk(G,\xx)$. The variety
$\shim_\kk(G,\xx)$ and the stack $\shim_\kk[G,\xx]$ both
descend to the reflex field $E(G,\xx)$.

More generally, let $\kk \subset G(\aff_f)$ be an open compact
subgroup, and let $\kk_0 \subset \kk$ be a neat subgroup of finite
index.  Define the corresponding Shimura stack by $\shim_\kk =
[\shim_{\kk_0}[G,\xx]/(\kk/\kk_0)]$; it is independent of the choice
of $\kk_0$.

Fix a prime $p$, let $v$ be a prime of $E(G,\xx)$ lying over $p$, and
let $\kk_p \subset G(\rat_p)$ be hyperspecial. Then the pro-variety
\[
\shim_{\kk_p} := \invlim{{\kk}^p\subset G(\aff^p_f)} \shim_{\kk_p {\kk}^p}(G,\xx)
\]
admits an  extension to $\calo_{E(G,\xx),v}$ as a pro-scheme
with continuous $G(\aff^p_f)$-action, which we continue to denote $\shim_{\kk_p}$.   What makes this model the
\emph{integral canonical model} is the following extension property:
If $T$ is any regular, formally smooth (pro\nobreakdash-)scheme over
$\calo_{E(G,\xx),v}$, then any morphism $T_{E(G,\xx)} \to \shim_{\kk_p}$
extends to all of $T$ (e.g., \cite[\S (2.3.7)]{kisin_intcan}.

Consequently, for any $\kk
\subset G(\aff_f)$ hyperspecial at $p$, $\shim_\kk[G,\xx]$ extends
canonically to a smooth Deligne-Mumford stack over
$\calo_{E(G,\xx),v}$.   (If necessary, one can start with the
canonical integral model of $\shim_{\kk_0}[G,\xx]$ for some neat
compact open subgroup $\kk_0 \subset \kk$, and then pass to the
quotient by the action of $\kk/\kk_0$.)

In fact, let $\kk \subset G(\aff_f)$ be an open compact subgroup, and
let $M = M(\kk)$ be the (finite) product of all primes $p$ such that the
component $\kk_p$ is \emph{not} hyperspecial.  Using
\cite[Thm.~2.2.1]{lovering17}, we see that $\shim_\kk[G,\xx]$
admits a canonical integral model over $\calo_{E(G,\xx)}[1/M]$.

A morphism $f:(G_1,\xx_1) \ra (G_2,\xx_2)$ of Shimura data is a morphism $G_1 \ra G_2$
of algebraic groups which induces a morphism $\xx_1 \ra \xx_2$.  For
future use, we collect some standard functorialities for morphisms of
Shimura varieties.

\begin{lemma}
\label{lemshimmorphisms}
Let $f:(G_1,\xx_1) \to (G_2,\xx_2)$ be a morphism of Shimura data.
Let $\kk_1 \subset G_1(\aff_f)$ and $\kk_2 \subset G_2(\aff_f)$ be
compact open subgroups such that $f(\kk_1) \subseteq f(\kk_2)$.
\begin{alphabetize}
\item Then $f$ induces a morphism $f_{\kk_1,\kk_2}:\shim_{\kk_1}[G_1,\xx_1] \to
  \shim_{\kk_2}[G_2,\xx_2]$ of Shimura stacks over $E$.
\item If $\kk_1$ and $\kk_2$ are hyperspecial at all $p\nmid M$, then
  $f_{\kk_1,\kk_2}$ extends to a morphism of Shimura stacks over
  $\calo_E[1/M]$.
\item If $f:G_1 \to G_2$ is injective, then $f_{\kk_1, \kk_2}$ is a
  closed morphism of Shimura stacks.  If $\kk_2$ is a sufficiently
  small compact open subgroup of $G_2(\aff_f)$ which contains $\kk_1$,
  then the generic fiber of $f_{\kk_1,\kk_2}$ is a closed embedding.
\end{alphabetize}
\end{lemma}

\begin{proof}
Part (a) and (the generic fiber of) part (c) are due to Deligne
\cite[1.15]{deligne_travaux}; see also \cite[5.16 and
13.8]{milne_shimura}.  The extension to integral models follows from the
extension property and the smoothness of the
integral model of $\invlim{\kk_1} \shim_{\kk_1}[G_1,\xx_1]$.
\end{proof}

\subsection{Orthogonal Shimura varieties}
\label{subsecorthshim}

Fix a nondegenerate lattice $L$ of signature
$(2,n_-)$, and let $G_L = \sorth_{L\tensor\rat}$ be the associated
special orthogonal group.  Let $\xx_L$ be the corresponding Hermitian symmetric space (\S \ref{subsecperioddomain}).

Inside $G_L(\aff_f)$ we single out the admissible integral automorphisms:
\[
\kk_L := \ker G_L(\widehat\integ) \to \aut(\disc(L))(\widehat\integ).
\]
The local component at $p$, $\kk_{L,p}$, is hyperspecial if
$p\nmid \Delta_L$.  Consequently, we have an integral canonical
model
\[
\shim_L := \shim_{\kk_L}[G_L,\xx_L]
\]
over $\integ[1/2\Delta_L]$.  (By inverting $2$, we sidestep the intricacies of orthogonal groups and Shimura varieties in even characteristic.)

Note that $\kk \cap \sorth_L(\real)^+ = \til{\sorth}_L(\integ)$, and thus
\cite[2.1.2]{deligne_corvallis}
\begin{equation}
\label{eq:shimlcx}
\shim_{L,\cx} \iso \sorth_L(\integ)^+
\backslash \xx_L^+ \iso \til\sorth_L(\integ)\backslash \xx_L.
\end{equation}

If $\kk \subset \kk_L \subset G_L(\aff_f)$ is any compact open
subgroup, then
there is a surjection $\shim_{\kk}[G_L,\xx_L] \to
\shim_{L}$ of stacks over $\integ[1/2\Delta_LN(\kk)]$. In particular, let $\kk_{L,N}  = \ker(\kk_L \to
G_L(\integ/N))$, and let $\shim_{L,N} = \shim_{\kk_N}[G_L,\xx_L]$, a
stack over $\integ[1/2\Delta_LN]$.  There is a surjection $\shim_{L,N}
\to \shim_L$, with each geometric fiber a torsor under
$\kk_L/\kk_{L,N} \iso G_L(\integ/N)$.

Now fix a primitive embedding of lattices $L_1 \inject
L_2$, with respective signatures $(2,n_{1,-})$ and $(2,n_{2,-})$.

\begin{lemma}
\label{lemshimclosed}
There is a closed morphism
\[
\xymatrix{
\shim_{L_1} \ar[r]^{\psi_{L_1,L_2}} & \shim_{L_2}
}
\]
of Shimura stacks over $\integ[1/(2\Delta_{L_1}\Delta_{L_2})]$
whose generic fiber is a closed embedding.
\end{lemma}

\begin{proof}
The chosen embedding gives an inclusion $G_{L_1} \to G_{L_2}$ of
groups over $\rat$, which induces $\xx_{L_1} \inject \xx_{L_2}$.
Because of the admissibility condition, we have an inclusion
$\kk_{L_1} \inject \kk_{L_2}$, whence (Lemma \ref{lemshimmorphisms}) a
morphism $\psi_{L_1,L_2}: \shim_{L_1} \to
\shim_{L_2}$ over $\integ[1/(2\Delta_{L_1} \Delta_{L_2})]$. To
verify that $\psi_{L_1,L_2,\rat}$ is a closed embedding, it suffices
to check that $\psi_{L_1,L_2,\cx}$ is an inclusion.  This last
claim follows from the description \eqref{eq:shimlcx} and the fact
that $\til\sorth_{L_1}(\integ)$ consists of those orthogonal
transformations of determinant one which lift to automorphisms of
$L_2$.
\end{proof}

Similarly, for each positive integer $N$, there is a closed morphism  $\shim_{L_1,N} \to \shim_{L_2,N}$ whose generic fiber is a closed embedding.

\subsection{Unitary Shimura varieties}
\label{subsecunitaryshim}

Let $K$ be a quadratic imaginary field.  Let $L$ be a free
$\calo_K$-module of rank $r$, equipped with a nondegenerate Hermitian form
$h(\cdot,\cdot)$ of signature $(1,r-1)$.  Attached to this is a
Shimura datum $(G_{\calo_K,L},\xx_{\calo_K,L})$, where $G_{\calo_K,L} = \U(L,h)$ is the group of
$\calo_K$-linear automorphisms of $L$ which preserve $h$, and
$\xx_{\calo_K,L} \iso \ball^{r-1}$, the
unit complex ball of dimension $r-1$.  Let $\kk_{\calo_K,L}$ be the
stabilizer in $G_{\calo_K,L}(\aff_f)$ of $L$.  Let
$\shim_{\calo_K,L}
=\shim_{\kk_{\calo_K,L}}[G_{\calo_K,L},\xx_{\calo_K,L}]$; it's the
moduli space of abelian varieties of dimension $r$ equipped with an
action by $\calo_K$ of signature $(1,r-1)$ and a polarization
$\lambda$ with $\ker(\lambda) \iso \disc(L)$. (More precisely, the relevant Shimura datum is
$(\U(L\tensor \rat,h), \xx_{(L\tensor \rat,h)})$; the choice of
lattice $L$ inside the $\rat$-vector space $L\tensor\rat$ defines the
integral structure on $G_{\calo_K,L}(\aff^f)$.)

More generally, suppose $K$ is a CM field, with maximal totally real
subfield $K^+$, and again let $L$ be a free $\calo_K$-module of
rank $r$, equipped with a nondegenerate Hermitian form $h$.  The archimedean
signature of $(L, h)$ is determined by data
\begin{equation}
\label{eq:defsig}
\st{(m_\sigma,n_\sigma)}_{\sigma:K^+\inject \real}.
\end{equation}
Let $G_{\calo_K,L} = \U(L,h)$; the associated Hermitian symmetric domain
$\xx_{\calo_K,L}$ has dimension $\sum m_\sigma n_\sigma$.  If there exists some
$\sigma_0$ such that $(m_{\sigma_0},n_{\sigma_0}) = (1,r-1)$, and
if $m_\sigma n_\sigma = 0$ for $\sigma \not = \sigma_0$, then we again
have $\xx_{\calo_K,L} \iso \ball^{r-1}$.  Again, let
$\shim_{\calo_K,L}  =
\shim_{\kk_{\calo_K,L}}[G_{\calo_K,L},\xx_{\calo_K,L}]$.

In the applications pursued here, it turns out that either $L$ is unimodular, or there
is some prime $p$ which is totally ramified in $\calo_K$; $\calo_K$
acts on $\disc(L)$ through its quotient $\ff_p$; and $\disc(L)$, as an
abelian group, is isomorphic to $(\integ/p)^{2\floor{(r-1)/2}}$.
This shows up in the analysis by Kudla and Rapoport \cite{kroccult} of
occult period maps.  The only impact on the present study is that it
shapes the structure of the polarization in the moduli-theoretic
intepretation of $\shim_{\calo_K,L}$.

In any event, the Shimura stack $\shim_{\calo_K,L}$ admits a smooth
integral model over $\calo_K[1/\Delta(K)\Delta_L]$.

\subsection{Shimura varieties and K3 surfaces}
\label{subsecshimk3}

Let $L\inject \lk$ be a primitive sublattice of signature $(1,r-1)$.
Consistent with earlier notation, we set
\[
\shim^L =
\shim_{\kk^L}[G^L,\xx^L]=\shim_{\kk_{L^\perp}}[G_{L^\perp},\xx_{L^\perp}]. \]

Now let $\ubar \chi = (\mmu_n,\chi^\omega,\chi)$ determine an action
of $\mmu_n$ on $L^\perp$ as in \ref{subseccatk3}.  Let $E(\ubar\chi) = \rat(\zeta_n)$.
\begin{itemize}
\item If $n \ge 3$, let $\shim^{(L,\ubar\chi)} = \shim_{\calo_{E(\ubar\chi)},L^\perp}$; then $\shim^{(L,\ubar\chi)}(\cx)$ is an arithmetic quotient of a complex ball.
\item If $n = 2$, let $\shim^{(L,\ubar\chi)} = \shim_{L^\perp}$; then $\shim^{(L,\ubar\chi)}(\cx)$ is an arithmetic quotient of a Hermitian symmetric space of type IV.

\end{itemize}

With this notation, we have:

\begin{lemma}
The periods of a structured K3 surface determine
holomorphic open immersions
\[
\xymatrix{
\mk_{L,\cx} \ar[r] ^{\tau_{L,\cx}} & \shim^L_\cx \\
\mk_{(L,\ubar\chi),\cx} \ar[r]^{\tau_{(L,\ubar\chi),\cx}} &
\shim^{(L,\ubar\chi)}.
}
\]
\end{lemma}

\begin{proof}
The period domain of a family of structured K3 surfaces is computed
in, e.g., \cite{dolgachev_lattice} and \cite{dolgachevkondo}. The
interpretation in terms of Shimura varieties is standard, and is drawn
out (in some cases) in, for instance,
\cite{kudlak3,perak3,rizovmoduli}.
\end{proof}

\begin{remark}
In the case of a datum $(L,\ubar\chi)$, the complement of the image of
$\tau_{(L,\ubar\chi)}$, when known, is often a ball quotient in its
own right; see, e.g., \cite{kondog4} for a representative example.
Kudla and Rapoport, in several cases, interpret this complement as a
"special cycle".  In particular, this complement is itself a
Shimura variety.  The author conjectures that this structure of the
complement holds integrally, as well.  In the special case of cubic
surfaces, this is worked out in \cite{achtercubics}; for now, it seems
that the general case remains open.
\end{remark}

\section{Integral period maps}

 With the notation established above, the Torelli theorem for complex
 K3 sufaces asserts that there is an open immersion
\[
\xymatrix{
\mk_{2d,\cx}  \ar@{^{(}->}[r]^{\tau_{2d,\cx}} &
\shim^{\ang{2d}}_\cx
}
\]
of stacks over $\cx$.  In fact, it is known that this map preserves
arithmetic:

\begin{proposition}
\label{proprizov}
The period map $\tau_{2d,\cx}$ descends to a morphism $\tau_{2d}:
\mk_{2d} \to \shim^{\ang{2d}}$ of stacks over $\rat$.
\end{proposition}

\begin{proof}
Rizov has proved this for $\mk_{2d}^\circ$, using an analogue of CM theory for K3 surfaces;
see \cite[p.14]{rizovcrelle} and \cite[Thm.~5]{taelman17}.  The
statement for $\mk_{2d}$ follows from descent relative to $\spec
\cx \to \spec \rat$, since $\mk_{2d}^\circ$ is dense in $\mk_{2d}$.
\end{proof}

Using Proposition \ref{proprizov} as a starting point, we will show
that other period maps also descend to a natural field of definition
and extend integrally.  We start with an interlude on level
structures, so that we can work with quasiprojective schemes and
verify descent in an elementary fashion.

\subsection{Level structures}
It is possible to define the notion of a lattice polarized K3 surface
with $\mathbb K$ level structure for an essentially arbitrary open
subgroup of $G^L(\widehat\integ)$; but we will content ourselves here with a
more limited notion which is adequate for our purposes.  (See \cite[\S
5.2]{taelman17} and \cite{rizovcrelle} for more details in the case $L =
\l2d$.)

Fix an integer $N>2$ which is relatively prime to $2p\Delta_L$.  Then
$\sorth^L(\integ/N)$ is admissble; any automorphism of $L^\perp\tensor
\integ/N$ lifts uniquely to
$\lk\tensor \integ/N$ as an element which fixes $L$.

If
$(Z \to S, \alpha) \in \mk_L(S)$, a full level $N$ structure on
$(Z\to S, \alpha)$ is an isomorphism of formed spaces $\beta: \lk
\tensor \integ/N \sriso R^2f_*\integ/N(1)$ such that the following
diagram commutes:
\[
\xymatrix{
\lk \tensor \integ/N \ar[r]^\beta_\sim & R^2 f_*\integ/N(1) \\
L\ar[u] \ar[r]^\alpha & \pic_{Z/S}(S) \ar[u]^{c_{1,N}}
}
\]
where the right-hand vertical map is the Chern class, and the
left-hand map is induced by the fixed inclusion $L\inject \lk$.

Since $N>2$, $\mk_{L,N}$ is representable by a smooth,
quasiprojective scheme over $\integ_{(p)}$ (see, e.g.,
\cite[Cor.~2.4.3]{rizovcrelle} for the case $L = \l2d$).  Moreover, because of the
admissibility condition,  $\mk_{L,N} \to \mk_L$ is Galois,
with covering group isomorphic to $\st{g \in \sorth_\lk(\integ/N): g
  \rest{L^\perp\tensor \integ/N} = \id}$.

As before, given $L$, choose a primitive embedding of lattices $\l2d
\inject L$.  The forgetful maps yield a Cartesian diagram
\[
\xymatrix{
\mk_{L,N} \ar@{^{(}->}[r] \ar[d] & \mk_{\l2d,N} \iso \mk_{2d,N} \ar[d]\\
\mk_L  \ar@{^{(}->}[r] & \mk_{\l2d}
}
\]
where the horizontal arrows are closed immersions, and the vertical
arrows are quotients by suitable subgroups of $\sorth_\lk(\integ/N)$.

\subsection{Descent to the reflex field}

\begin{lemma}
\label{lem:tauQ}
\begin{alphabetize}
\item Let $L \inject \lk$ be a primitive lattice of signature $(1,r-1)$.
Then the complex period map $\tau_{L,\cx}$ descends to a morphism
$\tau_{L}: \mk_L \to \shim^L$ of stacks over $\rat$.

\item Let $(L,\ubar\chi)$ be as in \ref{subseccatk3}.  Then the complex period map $\tau_{(L,\ubar\chi),\cx}$
descends to a morphism $\tau_{(L,\ubar\chi)}$ of stacks over $E(\ubar\chi)$.
\end{alphabetize}
\end{lemma}

\begin{proof}
We address part (a) in detail.
Fix some $N>2$.
Since $\mk_L = [\mk_{L,N}/G^L(\integ/N)]$ and $\shim^L =
[\shim^L_{N}/G^L(\integ/N)]$, it suffices to show that the complex
period map with level $N$ structure, $\tau_{L,N,\cx}: \mk_{L,N,\cx}
\to \shim^L_{N,\cx}$ descends to $\rat$.  Choose a primitive embedding
$\ang{2d} \inject L$.  We have a commuting diagram of universally injective morphisms of
complex reduced quasiprojective varieties
\begin{equation}
\label{diag:bootstrap}
\xymatrix{
\mk_{L,N,\cx} \ar@{^{(}->}[r]^{\tau_{L,N,\cx}}
\ar@{^{(}->}[d]^{\phi_{L,2d,\cx}} & \shim^L_{N,\cx}
\ar@{^{(}->}[d]^{\psi_{G^L,G^{\ang{2d}}}} \\
\mk_{2d,N,\cx}  \ar@{^{(}->}[r] & \shim^{\ang{2d}}_{N,\cx}.
}
\end{equation}
Since $\mk_{L,N,\cx} \to \mk_{2d,N,\cx}$ and $\mk_{2d,N,\cx}
\to \shim^{\ang{2d}}_{N,\cx}$ descend to $\rat$, so does
$\psi_{G^L,G^{\ang{2d}}} \circ \tau_{L,N,\cx}$.  Since
$\psi_{G^L,G^{\ang{2d}}}$ is universally injective (Lemma \ref{lemshimclosed}), $\tau_{L,N,\cx}$
is $\aut(\cx/\rat)$-equivariant on $\cx$-points, and thus descends to
$\rat$ as well.

The proof of (b) is exactly the same, except that
the role of $G^L(\integ/N)$ is now played by the finite unitary group
$G^{(L,\ubar\chi)}(\integ/N)$, and
\eqref{diag:bootstrap} is replaced with
\[
\xymatrix{
\mk_{(L,\ubar\chi),N,\cx}
\ar@{^{(}->}[r]^{\tau_{(L,\ubar\chi),N,\cx}}
\ar@{^{(}->}[d]^{\phi_{(L,\ubar\chi),L,\cx}} &
\shim^{(L,\ubar\chi)}_{N,\cx}
\ar@{^{(}->}[d]^{\psi_{G^{(L,\ubar\chi)},G^{L}}} \\
\mk_{L,N,\cx}  \ar@{^{(}->}[r] & \shim^{L}_{N,\cx}.
}
\]
\end{proof}

\subsection{Integral extension}

Granting the existence of integral canonical models of
Shimura varieties, it is not hard to see that $\tau_{2d}$ extends to a
morphism of stacks over $\integ[1/6d]$; this is achieved in
\cite[Thm.~4.3.3]{rizovmoduli}.  We refer to \cite[Cor.~5.15]{perak3}
for the difficult extension of this work to $\integ[1/2]$.

\begin{remark}
\label{rem:taelmantwist}
In fact, in
\cite{perak3}, the proof naturally gives rise to a period map for a
trivial double cover of $\mk_{2d}$.  Taelman has observed
\cite{taelman17} that by viewing the period map as measuring the
primitive cohomology \emph{twisted by the determinant}, the need for a
double cover is eliminated.

Following Taelman's analysis \cite[\S 5]{taelman17}, let $\kk_L^* =
\st{ \gamma \in \til\orth_L(\widehat\integ) : \det(\gamma) \in \st{\pm
    1}}$.  Then $\kk_{L^\perp}$ acts on $L^\perp$ via the determinant;
Taelman defines, for instance, a period map
\[
\xymatrix{
\mk_{2d}(\cx) \to [\sorth^{\ang{2d}}(\cx) \backslash (\xx^{\ang{2d}} \times \sorth^{2d}(\aff_f)/(\kk^{\ang{2d}})^*].
}
\]
The target space is isomorphic, as an analytic space, to our
$\shim_{\kk^{\ang{2d}}}[G^{\ang{2d}},\xx^{\ang{2d}}]$.  However, this
target naturally identifies a polarized Hodge structure $(H,s)$ with
$(H,-s)$.  In this way, the effect of a choice of generator of
$\ang{2d}$ is erased.

More generally, by following Taelman's formulation, we can suppress
the choice of a ``positive light cone'' in the definition  of $\mk_L$
in \ref{subseccatk3}; two $L$-polarizations which agree up to sign are
identified by the action of $\kk_L^*$ through its determinant.
\end{remark}

We now secure analogous results for other period maps.

\begin{lemma}
\label{lem:tauZ}
\begin{alphabetize}
\item Let $L\inject \lk$ be a primitive lattice of signature $(1,r-1)$.
Then the period map extends to a morphism $\tau_L: \mk_L \to
\shim^L$ of stacks over $\integ[1/2\Delta(L)]$.

\item Let $(L,\ubar\chi)$ be as in \ref{subseccatk3}.  Then the period map extends to a morphism
$\tau_{(L,\ubar\chi)}: \mk_{(L,\ubar\chi)} \to \shim^{(L,\ubar\chi)}$ of stacks
over $\calo_{E(\ubar\chi)}[1/2\Delta(L)]$.
\end{alphabetize}
\end{lemma}

\begin{proof}
Since $\shim^L$ is separated, it suffices to show
that, for a fixed $p\nmid 2\Delta(L)$, $\tau_L$ extends to
$\integ_{(p)}$.  Let $N\ge 3$ be a natural number relatively prime to
$p$.
Since $\mk_{L,N}$ is smooth (Proposition
\ref{propdeformlattice}), the extension
property of the integral canonical model implies that the morphism
$\tau_L$ extends to $\integ_{(p)}$.

The proof of (b) is the same, except that the necessary smoothness is
secured in Lemma \ref{lemdeformaut}.
\end{proof}

\begin{remark}
  The generic fiber of the morphism $\shim^L\inject \shim^{\ang{2d}}$ is a closed immersion (Lemma \ref{lemshimclosed}).  If it were known that $\psi_{L^\perp,\ang{2d}^\perp}$ is a closed immersion of Shimura stacks, one could give an elementary proof of Lemma \ref{lem:tauZ}, as follows.
Suppose $p\nmid \Delta(L)d(L)$; choose $d$ with $p\nmid d$ such that there exists a primitive $\ang{2d} \inject L$.
We start with a diagram as in \eqref{diag:bootstrap}, where all
objects are defined over $\integ_{(p)}$, except that  $\tau_L$ is only known to be defined over $\rat$:
\[
\xymatrix{
\mk_{L,N} \ar@{..>}[r]^{\tau_L} \ar[d]^\phi & \shim^L_N \ar@{^{(}->}[d]^\psi\\
\mk_{2d,N} \ar@{^{(}->}[r]^{\tau_{2d}}& \shim^{\ang{2d}}_N
}
\]
We know that $\phi$ is a closed immersion and $\tau_{2d}$ is an open immersion, and thus the composition $\mk_{L,N} \to \shim^{\ang{2d}}_N$ is a locally closed immersion.  All schemes involved are Noetherian and $\mk_{L,N}$ is reduced, so the image of $\mk_{L,N}$ is an open subscheme of a closed subscheme of $\shim^{\ang{2d}}_N$
\cite[\href{https://stacks.math.columbia.edu/tag/03DQ}{Tag
  03DQ}]{stacks-project}.  We are operating under the hypothesis  that $\psi$ is a closed
immersion (Lemma \ref{lemshimmorphisms}).  Since $\shim^L_N$ is reduced, $\psi$ maps $\shim^L_N$
isomorphically onto its image, a closed subscheme of $\stack
\shim^{\l2d}_N$.

We have observed (\S \ref{subsecshimk3}) that  $\tau_{2d}\circ \phi$ maps the
characteristic zero fiber $\mk_{L,N,\rat}$ into
$\psi(\shim^L_{N,\rat})$ inside $\shim^{\l2d}_N$.  Since
$\psi(\shim^L_N)$ is closed, $\mk_{L,N,\rat}$ is dense in $\mk_{L,N}$ (by flatness over $\integ_{(p)}$; see Proposition \ref{propmlsmooth}) and $\tau_{2d}\circ \phi(\mk_{L,N})$ is locally
closed, it follows that $\tau_{2d}\circ \phi_N(\mk_{L,N})
\subseteq \psi(\shim^L_N)$.

In particular, $\tau_{2d}\circ \phi$ factors through a locally
closed immersion $\tau_{L}: \mk_{L,N} \to \shim^L_N$. We again
invoke the fact that, for Noetherian schemes, a locally closed
immersion factors as an open immersion followed by a closed
immersion.  The fact that $\dim \mk_{L,N} = \dim \shim^L_N$ (and
the reducedness of $\mk_{L,N}$) now
implies that, in any such factorization, the closed immersion must be
the identity map and therefore $\tau_{2d}$ is an open immersion.

\end{remark}

\section{From complete intersections to K3 surfaces}
\label{sec:examples}

Thanks especially to works of Kond\=o, we know that sometimes one can
associate a structured K3 surface to certain types of complete
intersection varieties.  Some of these constructions are reviewed
here, with an eye towards making sense of these associations in
families, and ultimately explaining the arithmetic origin of Kond\=o's
analytic ball-quotient maps.

In an attempt to minimize repetition in the statement of our main
results, we make the following definition:

\begin{definition}
\label{defarithoccult}
Say that $(\mk^\circ, \stack N, \stack S, \kappa,\tau)$ satisfies $(\dagger)$
over $\calo$ if there is a diagram
\begin{equation}
\tag{$\dagger$}
\xymatrix{
\mk \ar[d]^\tau \ar[r]^\kappa & \stack N \\
\stack S
}
\end{equation}
of stacks over $\calo$ where $\kappa$ induces an isomorphism on coarse
moduli spaces, and $\tau$ induces an open immersion $\mk^\circ(\cx)
\inject \stack S(\cx)$.
\end{definition}

We should note that, in many of the examples studied here (\S
\ref{subsec:g4}, \ref{subsec:g3}, \ref{subsec:cubsurf}), Kudla and
Rapoport have already shown that a transcendentally-defined occult
period map descends to a natural cyclotomic field of definition
\cite[\S 9]{kroccult}.  Their method of proof goes back (at least) to
Deligne \cite[Thm.~2.12]{deligneniveau}.  Roughly speaking, one shows
that a monodromy representation is so large that a certain abelian
scheme admits no automorphisms, and thus descends.  This strategy
presumably also dispatches \S \ref{subsec:6points}, perhaps with
\cite{delignemostow} providing the necessary monodromy calculation.
Applications \S \ref{subsec:g6} and \ref{subsec:5points} don't
literally fit within the framework of unitary Shimura varieties
attached to quadratic imaginary fields, which may explain their
omission from \cite{kroccult}.

\subsection{Stacks of varieties with group action}

In Kond\=o's constructions, the original variety is encoded in the fixed
locus of the group action on the K3 surface.  If a group scheme $G/S$ acts on a scheme $Z/S$, one can define $Z^G$,
the fixed point stack \cite[Prop.~2.5]{romagny05}.

\begin{lemma}
\label{lemfixed}
  Suppose $Z \to S$ is a K3 space and $G \subset \aut_{Z/S}(S)$ is a
  nontrivial finite cyclic group.
  \begin{alphabetize}
  \item The fixed locus $Z^{G} \to S$ is a scheme.
  \item If $s\in S$, then $Z^{G}_s$ is smooth, and has at most
    one component of dimension one and genus at least two.

  \item If $S$ is irreducible with generic point $\eta$, and if
    $C_\eta  \subset Z^{G}_\eta$ is a curve of genus at least two, then the closure $C$ of
  $C_\eta$ in $Z^{G}$ is a smooth, proper relative curve over
  $S$.
  \end{alphabetize}
\end{lemma}

  \begin{proof}
    Since $Z\to S$ is an algebraic
    space, so is $Z^{G}$ \cite[Rem.~3.4(ii)]{romagny05}; since
    all components of all fibers are
    smooth (see below) of dimension at most one,  $Z^{G}$ is
    actually a scheme.

    The smoothness assertion of (b) is proved in \cite[Lemma
    2.2]{artebanisartitaki} (in characteristic zero) and
    \cite[Prop.~1.4]{keum14} (in positive characteristic).  The fact that there is
    at most one curve of general type is also in \cite[Lemma 2.2]{artebanisartitaki}.  (While the
    statement is only claimed for complex K3 surfaces, the argument
    relies on nothing more than the Hodge index theorem.)

    Part (c) follows from the upper semicontinuity, on $Z^{G}$,
    of the function $z \mapsto \dim(Z^{G}_{\varpi(z)})$ \cite[IV.13.1.3]{ega43}.
  \end{proof}

  We will occasionally have cause to work with the stacks of smooth
  relative uniform cyclic covers of projective spaces, as in
  \cite{arsievistoli}.  Recall that if $X \to S$ is a smooth scheme,
  then a smooth relative uniform cyclic cover of degree $n$ consists
  of a morphism $f:Y \to X$ which commutes with an action of $\mmu_n$
  on $Y$ such that the branch divisor of $f$ is smooth over $S$, and,
  Zariski-locally on $X$, $Y$ is $\mmu_n$-equivariantly isomorphic to
  $\calo_Y(U)[y]/(y^n-h)$.  With a slight adjustment of the notation
  of \cite{arsievistoli}, let $\stack H(n,m,d)$ be the stack of smooth
  relative uniform cyclic covers $f:Y \to P \to S$ of degree $n$,
  where $P \to S$ is a Brauer-Severi scheme of dimension $m$, and the
  branch divisor of $f$ has degree $d$.  Thus, for example, $\stack
  H(2,1,2g+2)$ is the moduli stack of hyperelliptic curves of genus
  $g$.  (A Brauer-Severi scheme $P \to S$ of dimension $m$ is an
  $S$-scheme which, \'etale-locally on $S$, is isomorphic to the
  projective space of dimension $m$.)

In the special case where $m=1$, let $\tilstack H(n,1,d)$ be the stack
of smooth relative uniform cyclic covers of Brauer-Severi curves
equipped with a labelling of the branch locus; there is a forgetful
map $\tilstack H(n,1,d) \to \stack H(n,1,d)$, with fiber a torsor under
the symmetric group $S_d$ on $d$ letters.  (Since a Brauer-Severi
scheme with a section is trivial, the underlying scheme of an object
in $\tilstack H(n,1,d)$ is actually a family of projective lines,
rather than merely \'etale-locally a family of projective lines.)

Let $\tilstack M_{0,d}$ be the moduli space of $d$  distinct, labelled
points in $\proj^1$.  By sending a labelled branched cover of the projective
line to its branch locus, we obtain a morphism $\tilstack H(n,1,d)
\to \tilstack M_{0,d}$.  In fact, this morphism is the rigidification
along $\mmu_n$; it factors as $\tilstack H(n,1,d) \to \tilstack
H(n,1,d)\fatslash \mmu_n \sriso \tilstack M_{0,d}$, and in particular
induces an isomorphism on coarse moduli
spaces.  This morphism is $S_d$-equivariant, and we have $\stack
H(n,1,d) \to \stack M_{0,d}$.

Below, we will often have a morphism $\alpha: \stack S \to \stack T$ of smooth
stacks.  Then each stack is normal, and in particular has a normal
coarse moduli space.  If $\alpha$ induces a bijection on geometric
points, then (by Zariski's main theorem) $\alpha$ induces an
isomorphism of coarse moduli spaces.

\subsection{Curves of genus four}
\label{subsec:g4}

Here we follow \cite{kondog4}.  The argument given here is also a
prototype for the remainder of this section.

Let $C/k$ be a smooth, projective nonhyperelliptic curve of genus $4$
with no vanishing theta constants, over an algebraically closed field
in which $6$ is invertible.  Its canonical model is the  (complete)
intersection in $\proj^3$ of quadric and cubic surfaces $Q$ and $S$.
Let $\varpi: Z\to Q$ be the triple cover of $Q$ branched along $C$;
then $Z$ comes equipped with an action by $\mmu_3$.
Let $M_1$ and $M_2$ be smooth lines  on $Q$ which represent the two
rulings, and let $N_i = \varpi\inv M_i$.  Then each $N_i$ is an
elliptic curve, and the two of them pair as $(N_1, N_2) = 3$.
Moreover, $N_1$ and $N_2$ span a primitive lattice of $\pic(Z)$,
isomorphic to $L_4 := U(3)$.
Let $L_4 \inject \lk$ be a primitive embedding; the orthogonal
complement of this copy of $U(3)$ is $L_4^\perp \iso U(3) \oplus U
\oplus E_8(-1)^{\oplus 2}$ \cite[p.~386]{kondog4}.
In the notation of \S
\ref{sec:lattice}, $d(L_4) = 3$, and so there is a closed immersion $\mk_L \inject \mk_{\ang 6}$ of smooth stacks over $\integ[1/6]$.

The action of $\mmu_3$ on $Z$ is nonsymplectic, in the sense that
$\chi^\omega$, the character of the representation of $\mmu_3$ by
which $\mmu_3$ acts on $H^0(Z,\Omega^2)$, is faithful.
Kond\=o explicitly
writes down a certain representation $\rho$ of $\mmu_3$ on $L_4^\perp$.
(Of course, $L_4^\perp$ is free over $\integ[\zeta_3]$, in accordance
with \cite[Lemma 1.1]{machidaoguiso98}.)
Let $\chi_4$
be the character of $\rho\oplus \rho_{\triv}^{\oplus
  2}$.
In the case where $k = \cx$, Kond\=o shows that $Z$ is an element of
$\mk_{L_4,\ubar\chi_4}(\cx)$.  Let $\stack N_4$ be the moduli space of
nonhyperelliptic curves of genus $4$.  It is not hard to extend the
work in \cite{kondog4} to show:

\begin{lemma}
\label{lemn4}
There is a morphism $\kappa_4: \mk_{L_4,\ubar\chi_4}^\circ \to \stack N_4$  of stacks
over $\integ[\zeta_3,1/6]$ which is a bijection on geometric points.
\end{lemma}

\begin{proof}
Suppose $(Z \to S, \alpha, \rho) \in \mk_{L_4,\ubar\chi_4}^\circ(S)$.
In particular, $Z \to S$ is an algebraic space.  Let $B = Z^{\mmu_3}
\to S$ be the scheme of fixed points (Lemma \ref{lemfixed}).

We will show that every  fiber of $B \to S$ is a smooth,
projective, nonhyperelliptic curve of genus $4$.  Then $B$ is a scheme
over $S$, and the sought-for
functor $\kappa_4:\mk_{L_4,\ubar\chi_4}^\circ \to \stack N_4$ is then given by
$(Z\to S,\alpha,\rho) \mapsto (Z^{\mmu_3} \to S)$.

So, let $s$ be a point of $S$.  If $s$ has residue characteristic
zero, the proof of \cite[Thm.~1]{kondog4} shows that $B_s$ is a
smooth, projective, nonhyperelliptic curve of genus $4$.

If $s$ has positive characteristic $p\ge 5$, since $\mk_{L_4,\ubar\chi_4}^\circ$ is smooth over $\integ[\zeta_3,1/6]$,
$s$ lifts to characteristic zero.  More precisely, there exist a
mixed characteristic discrete valuation ring $A$, with general and
special fibers $\eta$ and $\circ$, and a point $P \in
\mk^\circ_{L_4,\ubar\chi_4}(\spec A)$ with $P_\circ = s$.   The
characteristic zero result for $B_\eta$, combined with the
specialization argument of Lemma \ref{lemfixed}(c), shows there is a
(necessarily unique) smooth projective curve $C_s$ of genus $4$ in
$B_s$.

Moreover, the quotient $Z_s/\mmu_4$ is a quadric (cone)
\cite[p.~389]{kondog4}, and $C_s$ maps isomorphically onto its image
in the quotient.  Insofar as $C_s$ is a genus $4$ curve lying on a
quadric surface in $\proj^3$, it is not hyperelliptic.

This defines the morphism $\mk^\circ_{L_4,\ubar\chi_4} \to \stack N_4$.  Now
let $k$ be an algebraically closed field in which $6$ is invertible.  The
construction at the beginning of this subsection -- modified to take a
minimal resolution, if necessary, to account for the impact of
vanishing theta characteristics -- gives a set-theoretic section to
$\mk^\circ_{L_4,\ubar\chi_4}(k) \to \stack N_4(k)$.
\end{proof}

We can finally explain the arithmetic origin of Kond\=o's observation
that $\stack N_4(\cx)$ is an arithmetic ball quotient.

\begin{proposition}
\label{propn4}
The tuple $( \mk^\circ_{L_4,\ubar\chi_4},  \stack N_4,
\shim^{(L_4,\ubar\chi_4)}, \kappa_4, \tau_{L_4,\ubar\chi_4})$
satisfies $(\dagger)$ over $\integ[\zeta_3,1/6]$.
\end{proposition}

\begin{proof}
  This simply summarizes the foregoing.  Consider $\kappa_4$ from Lemma \ref{lemn4}.  Since it yields a bijection on
  geometric points, it induces
  an isomorphism of coarse moduli spaces.
For $\tau_{L_4,\ubar\chi_4}$,
  Kudla and Rapoport \cite[Thm.~8.1]{kroccult} interpret Kond\=o's
  isomorphism   \cite[Thm.~1]{kondog4} map as a morphism $\mk^\circ_{L_4, \ubar\chi_4,\cx} \to \shim^{(L_4,\ubar\chi_4)}_\cx$ (see \S \ref{subsecshimk3}).  Then
  Lemma \ref{lem:tauZ} shows that this map descends and spreads to
  $\integ[\zeta_3,1/6]$.
\end{proof}

\begin{remark}
\label{remKRg4}
In characteristic zero, Kudla and Rapoport use a transcendental
construction, and the fact that $\shim^{(L_4,\ubar\chi_4)}$ is a
moduli space for abelian varieties with action by $\integ[\zeta_3]$,
to interpret Kond\=o's construction as a morphism of stacks $\stack N_{4,\cx}
\to \shim^{(L_4,\ubar\chi_r)}_\cx$.  They then use a monodromy argument
  \cite[p.579]{kroccult} to
  show that this map descends to a morphism of stacks over
  $\rat(\zeta_3)$, and conjecture that it extends to a morphism over $\integ[\zeta_3]$.
\end{remark}

\begin{remark}
\label{rem:zheng}
Let $\mk^\circ_{L_4,\ubar\chi_4}\fatslash \mmu_3$ be the
rigidification of $\mk^\circ_{L_4,\ubar\chi_4}$ along $\mmu_3$
(\cite{abramovichcortivistoli}; see also \cite[\S 5]{romagny05}).  Then
$\mk^\circ_{L_4,\ubar\chi_4}\fatslash\mmu_3$ has the same coarse moduli
space as $\mk^\circ_{L_4,\ubar\chi_4}$, and the morphism of Proposition
\ref{propn4} factors as
\[
\xymatrix{
\mk^\circ_{L_4,\ubar\chi_4} \ar[r] & \mk^\circ_{L_4,\ubar\chi_4}\fatslash\mmu_3 \ar[r] & \stack N_4.
}
\]
In this notation, Kudla and Rapoport conjecture
\cite[Rem.~7.2]{kroccult} that the second map is an isomorphism of
stacks.  In fact, this has recently been resolved using transcendental
means by Zheng
\cite[Prop.~7.9]{zhengoccult}, who goes on to show that, at least over
$\cx$, there is an open immersion of orbifolds $\stack N_{4,\cx} \to
\shim^{(L_4,\ubar\chi_4)}_\cx$.  Zheng also proves analogous
statements in the situations of \S \ref{subsec:g3} and
\ref{subsec:cubsurf} below.
\end{remark}

\subsection{Curves of genus three}
\label{subsec:g3}

Kond\=o has given \cite{kondog3} a similar characterization of $\stack
N_3(\cx)$, the set of complex nonhyperelliptic curves of genus $3$.
This construction also descends to arithmetic geometry, as follows.

Let $C/k$ be a smooth, projective nonhyperelliptic curve of genus $3$
over an algebraically closed field in which $2$ is invertible.  Its
canonical model is a smooth, plane quartic curve; let $\varpi:Z \to
\proj^2$ be the cyclic quartic cover ramified along $C$. Then $Z$ is a
K3 surface, and inside
$\pic(Z)$ is a lattice $L_3 \iso A_1 \oplus A_1(-1)^{\oplus 7}$
\cite[p.~222]{kondog3}.  (Briefly, $Z$ is a double cover of $Y$,
itself a double cover of $\proj^2$ branched along $C$.  Then $Y$ is a
del Pezzo surface of degree two, and thus may also be obtained as the
blowup of a projective plane at seven points.  The seven copies of
$A_1(-1)$ in $\pic(Z)$ are obtained from lifts to $Z$ of the seven
exceptional divisors; the remaining element of modulus two is the
pullback of the class of a line on the projective plane.)  Then $L_3$
embeds primitively into $\lk$, with orthogonal complement $L_3^\perp
\iso U(2)^{\oplus 2} \oplus D_8(-1) \oplus A_1(-1)^{\oplus 2}$, and $d(L_3)=2$.

By construction, $Z$ comes equipped with an action by $\mmu_4$.  Then
$\mmu_4$ acts on the space of holomorphic two forms via a faithful
character, $\chi^\omega$.  The action $\rho$ of $\mmu_4$
on $L_3^\perp$ is given explicitly in \cite[\S 2]{kondog3}, and we  let $\chi_4$
be the character of $\rho \oplus \rho_{\triv}^{\oplus 8}$.

\begin{proposition}
\label{propn3}
There is a morphism $\kappa_3: \mk^\circ_{L_3,\ubar\chi_3} \to \stack N_3$ so
that $(\mk^\circ_{L_3,\ubar\chi_3}, \stack N_3, \shim^{(L_3,\ubar\chi_3)},
\kappa_3, \tau_{L_3,\ubar\chi_3})$ satisfies $(\dagger)$ over
$\integ[\sqrt{-1},1/2]$.
\end{proposition}

\begin{proof}
As in Lemma
\ref{lemn4}, the map $\kappa_3$ is given by $(Z\to S, \alpha,\rho)
\mapsto
(Z^{\mmu_4} \to S)$; \cite[p.~225]{kondog3} and the \'etale Lefschetz
fixed point theorem [SGA 5.III.(4.11.3)] provide the necessary
geometric input to show that $Z^{\mmu_4}$ is a relative
nonhyperelliptic curve of genus $3$.  The fact that $\kappa_3$ gives a
bijection on geometric points is
established by the construction at the beginning of this section.
The asserted behavior of  $\tau_{L_3,\ubar\chi_3}$ is a special case of
Lemma \ref{lem:tauZ}.
\end{proof}

\begin{remark}
See \cite[\S 7]{kroccult} for earlier results over $\rat(\sqrt{-1})$.
\end{remark}

\subsection{Curves of genus six}
\label{subsec:g6}

Following \cite{artebanikondo11}, let $\stack N_6$ denote the moduli
stack (over $\integ[1/2]$) of non-special curves of genus $6$.  (Thus, a curve is
represented by a point in $\stack N_6$ if it is smooth, projective and
irreducible of genus $6$, and neither hyperelliptic, trigonal,
bielliptic, nor smooth quintic and planar.)

In fact, let $C/k$ be a non-special curve over an algebraically closed
field.  The canonical embedding of $C$ is a quadric section of a
unique quintic del Pezzo surface $Y$ in $\proj^5$.  Let $Z$ be the
double cover of $Y$ branched along $C$.  (If $C$ has fewer than five
$g_6^2$'s, then one must actually take the minimal resolution of this
cover.)    Then $Z$ is a K3 surface with an action by
$\mmu_2 = \st{\pm 1}$, and this action fixes  a lattice in $\pic(Z)$
isomorphic to $L_6 := A_1\oplus A_1(-1)^4$ \cite[\S
2.1]{artebanikondo11}.  (Note that $d(L_6)=2$.)

\begin{proposition}
\label{propn6}
There is a morphism $\kappa_6:\mk^\circ_{L_6,\ubar\chi_6} \to \stack N_6$ such
that $(\mk^\circ_{L_6,\ubar\chi_6}, \stack N_6, \shim^{(L_6,\ubar\chi_6)},
\kappa_6, \tau_{L_6,\ubar\chi_6})$ satisfies $(\dagger)$ over
$\integ[1/6]$.
\end{proposition}

\begin{proof}
As before, $\kappa_6$ is given by sending $(Z\to S, \alpha,\rho) \in
\mk^\circ_{L_6,\ubar\chi_2}(S)$ to its fixed locus $Z^{\mmu_2} \to S$
(see \cite[p.~1452]{artebanikondo11} for the argument, valid in any
characteristic, that each geometric fiber $Z^{\mmu_2}_{\bar s}$ is a
non-special curve of genus $6$). The construction described  above gives
a set-theoretic section on geometric points, and
$\tau_{L_6,\ubar\chi_6}$ is supplied by Lemma \ref{lem:tauZ}.
\end{proof}

\subsection{Five points on a line}
\label{subsec:5points}

Kond\=o has also explained how, in favorable cases, one can associate a structured K3 surface to certain configuration spaces of points.

For instance, as in \cite{kondo5points}, consider the moduli space  $\tilstack M_{0,5}$ of five distinct, ordered points in $\proj^1$.  (In fact, our discussion extends to the case of \emph{stable} configurations of points.)

Fix an embedding $\beta: \proj^1\inject \proj^2$ as a coordinate line, and let $Q_\infty \in \proj^2$ denote a point ``at infinity'' which is not contained in $\beta(\proj^1)$.

Initially, let $k$ be an algebraically closed field, and let $(P_1,
\cdots, P_5) \in \tilstack M_{0,5}(k)$ be an ordered 5-tuple of distinct
points.  Following \cite[Sec.~3.1-3.2]{kondo5points}, let $C$ be the
cyclic degree five cover of $\proj^1$ ramified
exactly at $P_1, \cdots, P_5$.  It naturally admits a model as a plane
curve inside $\proj^2$, intersecting $\beta(\proj^1)$ exactly at
$Q_i := \beta(P_i)$ for $i = 1, \cdots,
5$. Let $L_i$ denote the line connecting $Q_i$ and $Q_\infty$; and let
$E_0 = \beta(\proj^1)$.

Next, let $X$ be the minimal resolution of the double cover of
$\proj^2$ branched along the sextic plane curve $C+E_0$, with covering
involution $\tau$.  The degree five automorphism of $C$ also
induces a degree five automorphism $\sigma$ of $X$.  Moreover, because
of our construction, we can identfy certain divisor (classes) on $X$
which are fixed by $\tau$.

Indeed, for $1 \le i \le 5$, there is an exceptional curve $E_i$ of
the minimal resolution of singularities corresponding to $Q_i \in
C\cap E_0$.  Moreover, the inverse image of  $L_i$ in $X$ is the union
of two smooth rational curves $F_{i,-}$ and $F_{i,+}$, which pass
(respectively) through the two points
$R_-$ and $R_+$ of $X$ lying over $Q_\infty$.  The involution $\tau$ exchanges
$F_{i,-}$ and $F_{i,+}$, and $\sigma$ stabilizes each of $F_{i,-}$ and
$F_{i,+}$.

Finally, there is a sixteenth tautological cycle on $X$, namely, the
inverse image $E_{0,X}$ of $E_0$ in $X$.  It is stable under $\sigma$
and $\tau$.

Let $L_5$ be the lattice generated by these sixteen divisors,
equipped with the intersection pairing.  Then $L_5 \iso V \oplus
A_4(-1) \oplus A_4(-1)$ \cite[Lemma
4.2]{kondo5points}, and $d(L_5)=2$.    The labelling of the original points, combined with the labelling of
the points in $X$ over $Q_\infty$, yields an inclusion $L_5 \inject
\pic(X)$.

There is an embedding $L_5 \inject \lk$. The
orthogonal complement of $L_5$ is computed in \cite[\S
4.3]{kondo5points}, and a structure $\rho$ of $L_5^\perp$ as a
$\mmu_5$-representation is described in \cite[\S 5.2]{kondo5points}.
Let $\chi_5$ be the character of $\rho\oplus \rho_{\triv}^{\oplus
  10}$.  Then $X$ is represented by a $k$-point of $\mk^\circ_{L_5,\chi_5}$.  Conversely, we have:

\begin{lemma}
\label{lem5points}
There is a morphism $\kappa_5:  \mk^\circ_{L_5,\ubar\chi_5} \to \tilstack
M_{0,5}$ so that $(\mk^\circ_{L_5,\ubar\chi_5},\tilstack M_{0,5},
\shim^{(L_5,\ubar\chi_5)}, \kappa_5,\tau_{L_5,\ubar\chi_5})$ satisfies
$(\dagger)$ over $\integ[\zeta_5,1/10]$.
\end{lemma}

\begin{proof}
As usual, it suffices to describe $\kappa_5$.  The
previous construction gives the desired inverse on geometric points,
and thus we have an induced isomorphism of coarse moduli spaces.

Suppose $(Z \to S, \iota,\alpha) \in \mk^\circ_{L_5,\ubar\chi_5}(S)$.
Then there is also an involution $\beta \in \aut_{Z/S}(S)$.  (In
characteristic zero, this is described in the last paragraph of the
proof of \cite[Lemma 5.7]{kondo5points}; in positive characteristic,
this then follows from a specialization argument.)  The fixed locus
$Z^{\beta}$ is the disjoint union of a curve $C \to S$ with each fiber
smooth and projective of genus $6$ (use \emph{loc. cit.} and Lemma
\ref{lemfixed}(c)) and a relative rational curve.  Moreover, the
action of $\mmu_5$ on $Z$ restricts to an action of $\mmu_5$ on $C$.
The lattice polarization, in particular the numbering of the cycles
$E_1, \cdots, E_5$, labels the fixed sections $C^{\mmu_5}\to S$.  The
quotient curve $C/\mmu_5 \to S$ has fibers of genus zero, and the
sought-for configuration is $(C^{\mmu_5} \subset C/\mmu_5) \in
\tilstack M _{0,5}(S)$.
\end{proof}

In this case, diagram $(\dagger)$ is part of a larger diagram of
moduli stacks.  By its construction, the map $\kappa_5:
\mk^\circ_{L_5,\ubar\chi_5} \to
\tilstack M_{0,5}$ factors through $\tilstack H(5,1,5)$.  Now, if $(C_S
\to\proj^1_S) \in \stack H(5,1,5)(S)$, then $\pic^0_{C/S}$ has an action
by $\integ[\zeta_5]$, of signature  \eqref{eq:defsig}
$\Sigma = \st{(2,1),(0,3)}$.  Inside $\stack A_6$ we have $\stack
A_{\integ[\zeta_5],\Sigma}$, the locus of principally polarized
abelian 6-folds with an action by $\integ[\zeta_5]$ of signature
$\Sigma$.  Consider the classical Torelli map
$\tau_6: \stack M_6 \to \stack A_6$.  The image of the restriction to
$\stack H(5,1,5)$ of $\tau_6$ is open in $\stack
A_{\integ[\zeta_5],\Sigma}$.

Of course, $\stack
A_{\integ[\zeta_5],\Sigma}$ is a Shimura variety in its own right. The
complex-analytic uniformization of $\stack A_{\integ[\zeta_5],\Sigma}$
is worked out in detail in \cite[Case (5)]{shimura64}.   Let $G =
G_{\integ[\zeta_5],L_5^\perp}$, and let $\xx_G$ be the corresponding
Hermitian symmetric domain; it is isomorphic to the unit $2$-ball
$\mathbb B^2$.  There is a compact open subgroup $\kk_0 \subset
G(\aff_f)$ such that $\stack A_{\integ[\zeta_5],\Sigma}\iso
\shim_{\kk_0}[G,\xx_G]$.

(Briefly, let $M = \integ[\zeta_5]^{\oplus
  3}$, endowed with the Hermitian form $h$ represented by
$\operatorname{diag}(1,1, \half{1-\sqrt 5})$.  The unitary group of
$(M,h)$ is an integral form of $G$, and $\kk_0$ is the stabilizer of
the lattice $M$.  Conversely, $\kk^{(L_5,\ubar\chi_5)}$ can be
recovered from $\kk_0$ as those group elements which act trivially on
the discriminant group of $L$.)

   Then $\kk_0 \supset
\kk^{(L_5,\ubar\chi_5)}:=\kk_{\integ[\zeta_5],L_6^\perp}$, with
quotient group $\kk_0/\kk^{(L_5,\ubar\chi_5)} \iso
\orth(\disc(L^\perp)) \iso \orth_3(\ff_5) \iso \st{\pm 1} \times S_5$.

We summarize this discussion in:

\begin{proposition}
There is a diagram of stacks over $\integ[\zeta_5,1/10]$:
\[
\xymatrixcolsep{5pc}
\xymatrix{
\mk^\circ_{L_5,\ubar\chi_5} \ar@/^3ex/^{\kappa_5}[rr] \ar[r] \ar[dd]^{\tau_{L_5,\ubar\chi_5}} &
\tilstack H(5,1,5) \ar[r] \ar[d]^{[S_5]} &
\tilstack M_{0,5} \ar[d]^{[S_5]} \\
&\stack H(5,1,5)\ar[r] \ar[d]^{[\mmu_2]} & \stack M_{0,5} \\
\shim^{(L_5,\ubar\chi_5)} \ar[r]^{[\kk_0/\kk^{(L_5,\ubar\chi_5)}]} &
\stack A_{\integ[\zeta_5],\Sigma} &
}
\]
where an arrow is labelled $[\Gamma]$ if it is a quotient by the
finite group $\Gamma$; the given factorization of $\kappa_5$ is, on
coarse moduli spaces, a composition of isomorphisms; and
$\tau_{L_5,\ubar\chi_5,\cx}$ is an open immersion.
\end{proposition}

\begin{proof}
Since the canonical models of both $\shim^{(L_5,\ubar\chi_5)}$ and $ \stack A_{\integ[\zeta_5],\Sigma}$ receive maps from $\stack R_{L_5,\ubar\chi_5}$ over $\integ[\zeta_5,1/10]$ with dense image, it suffices to observe that the quotient map $\shim^{(L_5,\ubar\chi_5)} \to \stack A_{\integ[\zeta_5],\Sigma}$ is defined on the canonical models.
\end{proof}

\subsection{Six points on a line}
\label{subsec:6points}

In \cite[\S 12]{dolgachevkondo}, Dolgachev and Kond\=o show that the configuration space of six labelled points on the (complex) projective line is an arithmetic quotient of $\ball^4$.

The pointwise construction of \emph{loc. cit.} works over an arbitrary
algebraically closed field $k$.  Let $(P_1, \cdots, P_6) \in \tilstack
M_{0,6}$ be an ordered $6$-tuple of distinct points.  Let $C$ be the
cyclic degree three cover of $\proj^1$ ramified exactly at the $P_i$,
and let $Z'$ be the cyclic triple cover of the ambient weighted
projective space $\proj(1,1,2)$   ramified along
$C$.  (Explicitly, let $f(X_0,X_1)$ be a homogeneous form of degree
$6$ vanishing at the $P_i$; then $Z'$ is given by the equation
$X_3^3+X_2^3 + f(X_0,X_1)=0$ in $\proj(1,1,2,2)$.)  Then $Z'$ comes with an action by $\mmu_3 \times \mmu_3$; we
single out the action of $\mmu_3$ on $Z'$ via the diagonal embedding
$\mmu_3 \inject \mmu_3\times \mmu_3 \inject \aut_k(Z')$. The variety
$Z'$ has three ordinary nodes; its minimal resolution, $Z$, is a K3
surface, and the $\mmu_3$ action lifts to $Z$.  One finds that, by
construction, $\pic_{Z/k}(k)$ comes equipped with a primitive
inclusion of the lattice $L'_6 := U \oplus E_6(-1) \oplus
A_2(-1)^{\oplus 3}$, with orthogonal complement $A_2(1) \oplus
A_2(-1)^{\oplus 3}$, and that $d(L'_6)=2$. As a $\integ[\zeta_3]$-module, $(L'_6)^\perp$ is
free of rank $4$, and comes equipped with a Hermitian form of
signature $(3,1)$.  Let $\rho$ be the corresponding
$\mmu_3$-representation, and let $\chi'_6$ be the character of $\rho
\oplus \rho_{\triv}^{\oplus 14}$.  As usual, we have $(\dagger)$ for
$\mk^\circ_{(L'_6,\ubar\chi'_6)}$, $\tilstack M_{0,6}$ and
$\shim^{(L'_6,\ubar\chi'_6)}$ over  $\integ[\zeta_3,1/6]$.

Alternatively, we could use the strategy of \S \ref{subsec:5points}, and
compute the periods of $C$ directly.  If $(C \to S \to \proj^1_S) \in
\stack H(3,1,6)$, then $C/S$ is a family of curves of genus $4$, and $\pic^0_{C/S}$ has an action by
$\integ[\zeta_3]$ of signature $(1,3)$.  (This is case (2) of
\cite{shimura64}.)  The moduli space $\stack A_{\integ[\zeta_3],(1,3)}$
of principally polarized abelian fourfolds with action by
$\integ[\zeta_3]$ of signature $(1,3)$ is isomorphic to
$\shim_{\kk_0}[G,\xx_G]$, where $G = G_{\integ[\zeta_3],(L_6')^\perp}$
and $\kk_0$ is the stabilizer of the lattice $(L'_6)^\perp$.  There is
a surjection $\shim^{(L'_6,\ubar\chi'_6)} \to \shim_{\kk_0}[G,\xx_G]$
with covering map $\kk_0/\kk^{(L'_6,\ubar\chi_6)}\iso
\orth(\disc((L'_6)^\perp)) \iso \mmu_2 \times S_6$ (it seems that, in the
third displayed equation of \cite[p.93]{dolgachevkondo}, the authors
may have neglected to account for the discriminant kernel) and we
obtain:

\begin{proposition}
There is a diagram of stacks over $\integ[\zeta_3,1/6]$:
\[
\xymatrixcolsep{5pc}
\xymatrix{
\mk^\circ_{L_6',\ubar\chi'_6} \ar@/^3ex/^{\kappa'_6}[rr] \ar[r] \ar[dd]^{\tau_{L_6',\ubar\chi'_6}} &
\tilstack H(3,1,6) \ar[r] \ar[d]^{[S_6]} &
\tilstack M_{0,6} \ar[d]^{[S_6]} \\
&\stack H(3,1,6)\ar[r] \ar[d]^{[\mmu_2]} & \stack M_{0,6} \\
\shim^{(L_6',\ubar\chi'_6)} \ar[r]^{[\kk_0/\kk^{(L_6',\ubar\chi'_6)}]} &
\stack A_{\integ[\zeta_3],(1,3)} &
}
\]
where an arrow is labelled $[\Gamma]$ if it is a quotient by the
finite group $\Gamma$; the given factorization of $\kappa_6'$ is, on
coarse moduli spaces, a composition of isomorphisms; and $\tau_{L_5,\ubar\chi_5,\cx}$ is
an open immersion.
\end{proposition}

\subsection{Cubic surfaces}
\label{subsec:cubsurf}

Let $\mcs$ be the moduli space of cubic surfaces. If $V/\cx$ is a
complex cubic surface, then either by associating a cubic threefold
\cite{actsurfaces} or a K3 surface \cite{dvgk} to it and measuring its
periods, one obtains an open immersion $\mcs(\cx) \inject
\ball^4/\Gamma$.  The arithmetic nature of this map is explored in
\cite{achtercubics}.  Unfortunately, there is a stack-theoretic
mistake there. While $\mcs$ and  $\stack
H(3,3,3)$ have isomorphic coarse moduli spaces, they are not literally
the same stack; $\stack H(3,3,3) \to \mcs$ is the rigidification along
the $\mmu_3$-action.  We take the present opportunity to correct this
oversight, and recast the main result in the framework developed here.

Let $\mct$ be the moduli space of smooth projective cubic
\emph{threefolds}.  If $T \in \mct(\cx)$, then its intermediate
Jacobian is a principally polarized abelian fivefold.  This gives a
period map $\mct(\cx) \to \stack A_5(\cx)$, which is known to be an
embedding.  By using either monodromy considerations
\cite[Thm.~2.12]{deligneniveau} or the arithmetic nature of intermediate
Jacobians \cite[Thm. 6.1]{acmvdmij}, one can show that this period map
descends to a morphism $\mct \to \stack A_5$ of stacks over $\rat$.
Using the algebro-geometric construction of the intermediate Jacobian
as a Prym, one can actually spread out this out to achieve a morphism
$\mct \to \stack A_5$ of stacks over $\integ[1/2]$
\cite[Cor.~3.5]{achtercubics}.

Now, points of $H(3,3,3)$ correspond to cyclic triple covers of
$\proj^3$ branched along a cubic surface; as such, they are smooth
projective threefolds in their own right.  The $\mmu_3$ action on the
threefold induces an action of $\integ[\zeta_3]$ on the corresponding
intermediate Jacobian, with signature $(1,4)$.  Ultimately, one
obtains
\begin{proposition}
There is a diagram of stacks over $\integ[\zeta_3,1/6]$
\[
\xymatrix{
\mct\ar[d]  &\stack H(3,3,3) \ar@{_{(}->}[l] \ar[r]^\kappa \ar@{_{(}->}[d]^\tau  & \mcs \\
\stack A_5 & \stack A_{\integ[\zeta_3],(1,4)} \ar@{_{(}->}[l]
}
\]
in which $\tau$ is an open immersion, and $\kappa$
induces an isomorphism of coarse moduli spaces.
\end{proposition}

\bibliographystyle{hamsplain}
\bibliography{k3}

\end{document}